\documentclass[reqno]{amsart}
\usepackage[utf8]{inputenc}
\usepackage{enumerate}
\usepackage{amsmath, amssymb,amsthm}
\usepackage{mathtools}
\usepackage{leftidx}
\usepackage[font=small,skip=0pt]{caption}
\numberwithin{equation}{section}
\usepackage[left=3cm, right=3cm]{geometry}
\usepackage{hyperref}
\usepackage{xcolor}
\definecolor{my-black}{rgb}{0,0,0}
\definecolor{my-blue}{rgb}{0,0,0.8}
\definecolor{my-red}{rgb}{0.8,0,0}
\hypersetup{colorlinks,
    linkcolor	= {my-blue},
    citecolor	= {my-blue},
    urlcolor	= {my-blue}
}
\usepackage{graphicx}
\usepackage{tikz-cd}

\theoremstyle{plain} 

\newtheorem{lemma}{Lemma}[section]
\newtheorem{theorem}{Theorem}
\newtheorem*{theorem-with-no-label}{Theorem}
\newtheorem{corollary}{Corollary}[section]
\newtheorem{proposition}{Proposition}[section]

\theoremstyle{definition} 
\newtheorem{remark}{Remark}[section]

\theoremstyle{remark}
\newtheorem*{remark-with-no-label}{Remark}

\DeclareMathOperator{\real}{Re}
\DeclareMathOperator{\imag}{Im}

\DeclareMathOperator{\covol}{covol}
\DeclareMathOperator{\Tr}{Tr}

\DeclareMathOperator{\Nr}{N}

\DeclareMathOperator{\End}{End}
\DeclareMathOperator{\Hol}{Hol}

\DeclareMathOperator{\SL}{SL}

\DeclareMathOperator{\PSL}{PSL}

\DeclareMathOperator{\Or}{O}

\DeclareMathOperator{\sgn}{sgn}
\DeclareMathOperator{\disc}{disc}

\newcommand{\sm}{\smallsetminus}

\newcommand{\R}{\mathbb{R}}
\renewcommand{\H}{\mathbb{H}}
\newcommand{\C}{\mathbb{C}}
\newcommand{\N}{\mathbb{N}}
\newcommand{\Q}{\mathbb{Q}}
\newcommand{\Z}{\mathbb{Z}}

\renewcommand{\d}{\mathfrak{d}}
\renewcommand{\a}{\mathfrak{a}}
\renewcommand{\b}{\mathfrak{b}}

\newcommand{\Oo}{\mathcal{O}}
\newcommand{\Ss}{\mathcal{S}}
\newcommand{\Rr}{\mathcal{R}}
\newcommand{\Ii}{\mathcal{I}}
\newcommand{\Jj}{\mathcal{J}}

\title{Fourier non-uniqueness  sets from totally real number fields }

\author[Danylo Radchenko]{Danylo Radchenko}
\address{ETH Z{\"u}rich, Mathematics Department, R\"amistrasse 101, 8092 Z\"urich, Switzerland}
\email{danradchenko@gmail.com}

\author[Martin Stoller]{Martin Stoller}
\address{B{\^a}timent des Math{\'e}matiques, EPFL, Station 8, CH-1015 Lausanne, Switzerland}
\email{marstoller@gmail.com}

\date{\today}

\begin{document}

\begin{abstract}
Let $K$ be a totally real number field of degree $n \geq 2$. The inverse different of $K$ gives rise to a lattice in $\R^n$. We prove that the space of Schwartz Fourier eigenfunctions on $\R^n$ which vanish on the
``component-wise square root" of this lattice, is infinite dimensional. 
The Fourier non-uniqueness  set thus obtained is a discrete subset of the union of all spheres $\sqrt{m}S^{n-1}$ 
for integers $m \geq 0$ and, as $m \rightarrow \infty$,
there are $\sim c_{K} m^{n-1}$ many points on the $m$-th sphere for some explicit
constant $c_{K}$, proportional to the square root of the discriminant of $K$. This
contrasts a recent Fourier uniqueness result by Stoller \cite[Cor. 1.1]{S}. Using a different construction involving the codifferent of $K$, we prove an analogue for discrete subsets of ellipsoids. In special cases, these sets also lie on spheres with more densely spaced radii, but with fewer points on each.

We also study a related question about existence of Fourier interpolation formulas with nodes ``$\sqrt{\Lambda}$'' for general lattices  $ \Lambda \subset \R^n$. Using results about lattices in Lie groups of higher rank we prove that if $n \geq 2$ and a certain group $\Gamma_{\Lambda} \leq \PSL_2(\R)^n$ is discrete, then such interpolation formulas cannot exist.    Motivated by these more general considerations, we revisit the case of one radial variable 
and prove, for all $n \geq 5$ and all real $\lambda > 2$, Fourier interpolation 
results for sequences of spheres $\sqrt{2 m/ \lambda}S^{n-1}$, where $m$ ranges over 
any fixed cofinite set of  non-negative integers. The proof relies on a series of 
Poincar{\'e} type  for Hecke groups of infinite covolume and is similar to the one in 
\cite[\S 4]{S}. 
\end{abstract}

\maketitle
\tableofcontents

\section{Introduction}\label{sec:introduction}

The subject of this paper is motivated by recent work on Fourier uniqueness and
non-uniqueness pairs.  Broadly speaking, we are interested in the following general
question. Given a space $V$ of continuous integrable functions on $\R^n$ and two
subsets $A, B \subset \R^n$, when is it possible to recover any
function $f \in V$ from the restrictions $f|_A$ and $\widehat{f}|_B$
(where~$\widehat{f}$ denotes the Fourier transform of~$f$)?
In other words, we are interested in conditions on $A,B,V$, under which the
restriction map $f \mapsto (f|_A, \widehat{f}|_B)$ is injective. When
the map is injective, we say that $(A, B)$ is a (Fourier) uniqueness pair and if $A =
B$, we simply say that~$A$ is a (Fourier) uniqueness set. Conversely,
if the map is not injective, we call $(A,B)$ a non-uniqueness pair, and when $A=B$
we call $A$ a non-uniqueness set. Naturally, one would like the function space~$V$ to
be as large as possible and the sets~$A$ and~$B$ to be as small as possible, or
``minimal" in a certain sense.

A prototypical example of a minimal Fourier uniqueness set was found by Radchenko and
Viazovska in~\cite{RV}, where they proved that, when
$V = \Ss_{\text{even}}(\R)$ is the space of even Schwartz functions on the real line, 
the set $A = \sqrt{\Z_{+}}
:= \{ \sqrt{n}\,:\, n \in \Z_{\geq 0}\}$ is a uniqueness set and
established an interpolation theorem in this setting. The result is sharp
in the sense that no proper subset of $A$ remains a uniqueness set for
$\Ss_{\text{even}}(\R)$. Their proof was based
on the theory of classical modular forms, which is also well-suited to treat the case
$V = \Ss_{\text{rad}}(\R^n)$ of \emph{radial} Schwartz functions on $\R^n$ and the set $A =U_n := \cup_{m \in \N_0}{\sqrt{m}S^{n-1}}$. For the latter generalization, we refer to \S 2 in \cite{RS}, which deduces the result from~\cite{BRS}.

The second author recently proved  an interpolation formula \cite[Thm 1]{S} generalizing the one by Radchenko--Viazovska also to non-radial functions, that is, to the space $V = \Ss(\R^n)$ and the same set of concentric spheres $U_n$.  However, for $n>1$, it is no longer minimal. 
Indeed, the (related) interpolation formula in~\cite[Eq.~(4.1)]{RS} implies that
the space of $f \in \Ss(\R^n)$ satisfying $f(x)=\widehat{f}(x)=0$ for all $x\in 
\cup_{m \ge N}{\sqrt{m}S^{n-1}}$ is finite-dimensional for all $N$ and is in fact contained in 
$\mathcal{H}_{4N+2}\otimes W$ for some finite-dimensional space $W\subset 
\Ss_{\text{rad}}(\R^n)$, where $\mathcal{H}_k$ denotes the space of harmonic 
polynomials on $\R^n$ of degree $\le k$. Since a generic subset of $\dim 
\mathcal{H}_k$ points in $rS^{n-1}$ is an interpolation set for the space 
$\mathcal{H}_k$ (in the sense that any polynomial $p\in \mathcal{H}_k$ is uniquely 
determined by its values on $\dim \mathcal{H}_k$ generic points), this implies that 
there is a uniqueness set properly 
contained in $U_n$ that contains only finitely many points on spheres with 
radius $\le \sqrt{N}$.

In fact, it was recently proved by the second author and Ramos in \cite[Rmk 4.1, Cor 4.1]{RS} that \emph{any} discrete and sufficiently  uniformly distributed  subset $D \subset U_n$ remains a uniqueness set for $\Ss(\R^n)$. Here, ``sufficiently" means that $D \cap \sqrt{m}S^{n-1}$ contains at least $  C_n m^{c_n m}$ many points.  

We contrast these Fourier uniqueness results by providing two families of discrete \emph{non}-uniqueness sets in $\R^n$, where one of them is again contained $U_n$, while the other lies in a union of ellipsoids. Both of them are constructed from lattices corresponding to ideals in totally real number fields $K/\Q$ of degree $n$ and their density grows with the discriminant of $K$ (although their distribution is not uniform, in the sense that they ``avoid" points near the coordinate axes; see Figure \ref{fig1}). We give the precise formulations in the next subsection. Thus, characterizing the \emph{discrete} Fourier uniqueness sets contained in $U_n$ seems to be a subtle question. 

In fact, the motivation for this work was not to find negative results in this direction, but to try to generalize the modular form theoretic approach of Radchenko and
Viazovska to treat not necessarily radial functions on $\R^n$, in a way that is very different from the approach taken by Stoller (who essentially reduces the problem again to the case of radial Schwartz functions). More specifically, we were interested in (possible) interpolation formulas where we replace the set of nodes $A=\sqrt{\Z_{+}}$ by ``square roots'' of certain lattices coming from totally real number fields $K$, specifically, the co-different $\Oo_{K}^{\vee}$ of their ring of integers $\Oo_{K}$. In this set up, it seemed natural to ask whether one could be working with Hilbert modular forms and associated integral transforms, similarly to the proof by Radchenko--Viazovska.

As we will explain more in~\S\ref{sec:algebraic-obstructions} and briefly in \S \ref{subsec:general-lattices-in-intro}, there is an obstruction to the existence of
such interpolation formulas. From the more general point of view taken in \S \ref{sec:algebraic-obstructions}, the obstruction arises because, for $n \geq 2$, 
subgroups of $\PSL_2(\R)^n$ that are commensurable to the Hilbert modular group 
$\PSL_2(\Oo_K)$ are irreducible lattices and can therefore never contain subgroups of finite index with infinite abelianization, by Margulis' normal 
subgroup theorem. On the other hand the presence of 
certain unfavorable relations in the Hilbert modular group can be exploited in an 
explicit manner to obtain the non-uniqueness sets indicated in the abstract.
\subsection{Statement of non-uniqueness results}
We prepare for the formulation of our main non-uniqueness results and at the same time, introduce
some notation that will be used throughout the paper. Let $K$ be a totally real
number field of degree $n \geq 2$ with ring of integers $\Oo_{K}$.   Some of the
objects we will introduce depend on the number field $K$, but we will not always
display this dependence in our notation.

 We denote the $n$ real embeddings by $\sigma_j : K \rightarrow \R$, $1 \leq j \leq
 n$ and assemble them into the map $\sigma : K \rightarrow \R^n$, $\sigma(x) = (\sigma_1(x), \dots, \sigma_n(x))$. We recall that the
 trace of an element $x \in K$ is given by $\Tr(x) = \Tr_{K/
 \Q}(x)=\sum_{j=1}^{n}{\sigma_j(x)}$. For any $\Oo_{K}$-submodule $\a \subset K$ we
 write
\[
 \a^{\vee}  = \left \lbrace  x \in K \,:\,  \, \Tr_{K/\Q}(ax) \in \Z\,\; \mbox{for
 all } a \in \a \right \rbrace
\]
for its dual with respect to the trace paring.  As is well-known, if $\a$ is a
fractional ideal in $K$, then $\sigma(\a) \subset \R^n$ is a lattice  and
$\sigma(\a^{\vee}) =\sigma(\a)^{\vee}$, where on the right we mean the dual lattice
in the usual sense. Moreover, the covolume of $\sigma(\a)$ is given by
\begin{equation}\label{eq:formula-covolume-fractional-ideal}
\covol{(\sigma(\a))} = \Nr(\a) \sqrt{|\disc(K)|},
\end{equation}
where $\Nr(\a) \in \Q_{>0}$ is the ideal norm of $\a$ (the unique extension of the absolute norm on integral ideals to all fractional ideals of $K$) and 
$|\disc(K)| =
\covol(\sigma(\Oo_{K}))^2$ is the discriminant of~$K$.  For any fractional ideal $\a \subset K$ we define
\begin{equation}\label{eq:def-square-root}
\sqrt{\a}:= \sqrt{\sigma(\a)}:= \left \{  (x_1, \dots, x_n) \in \R^n \,:\, (x_1^2, 
\dots, x_n^2) = 
\sigma(\alpha)\, \text{ for some } \alpha \in \a \right \} \subset \R^n
\end{equation}
(which is not to be confused with the radical of an ideal). Recall that the codifferent (or inverse different) of $K$  is the fractional ideal
$\Oo_{K}^{\vee}$ and that the different $\d= \d_{K}$ is defined as $\d=
(\Oo_{K}^{\vee})^{-1}$. We see that the points $\sqrt{\Oo_{K}^{\vee}}$ lie on spheres
$\sqrt{m}S^{n-1}$ with non-negative integers $m$, the traces of the totally
non-negative elements in $\Oo_{K}^{\vee}$ (recall that an element $x \in K$ is said
to be totally non-negative
(resp. totally positive) if $\sigma_j(x) \geq 0$  (resp. $\sigma_j(x) > 0$) for 
all~$j=1,\dots,n$). We return to these points in \S \ref{subsec:number-of-points}.

For $f \in L^1(\R^n)$ we normalize its Fourier transform by $\widehat{f}(\xi) =
\int_{\R^n}{f(x) e^{- 2\pi i \langle x, \xi \rangle}dx}$, $\xi \in \R^n$. We
sometimes also use the notation $\mathcal{F}(f) = \mathcal{F}_{\R^n}(f) =
\widehat{f}$. Finally, we write $\H = \{ z \in \C\,:\, \imag(z) > 0 \}$ for the upper
half plane.
\begin{theorem}\label{thm:main-thm-spheres}
Let $K$ be  a totally real number field of degree $n \geq 2$ as above. Let $V \subset
\Ss(\R^n)$ denote the subspace linearly spanned by all Gaussians $e^{\pi i
z_1 x_1^2} \cdots e^{\pi i z_n x_n^2}$ with $z_j \in \H$, $x_j \in \R$.  Then for
any $\epsilon \in \{ \pm 1 \}$ the subspace of all $f \in V$ satisfying $f(x) = 0$ 
for all $x \in \sqrt{\Oo_{K}^{\vee}}$ and $\widehat{f} = \epsilon f$ is infinite 
dimensional.
\end{theorem}

\begin{remark} \label{rem:finiteX}
Since the space of Fourier eigenfunctions vanishing on $\sqrt{\Oo_{K}^{\vee}}$ is 
infinite-dimensional we can obtain nontrivial 
functions vanishing in addition on an arbitrary finite subset of $\R^n$, by a simple linear algebra argument. A similar remark applies to Theorem \ref{thm:main-thm-ellipsoids} below.
\end{remark}

Besides points on  spheres $\sqrt{m}S^{n-1}$, our methods also allow us to treat
other sets related to the different, which in general lie on ellipsoids. To formulate
it, we appeal to a Theorem of Hecke \cite[\S 63, Satz 176]{H2}, asserting that the
different $\d$ defines a square in the ideal class group of $K$. This means that we
can choose a fractional ideal $\a$ and a scalar $c \in K^{\times}$ such that
\begin{equation}\label{eq:inverse-different-as-square}
c \a^2 = \d^{-1}.
\end{equation}
Let us then define the set
\begin{equation}\label{eq:def-set-E_F}
E(c, \a):= \left \{ x \in \R^n \,:\, \exists \alpha \in \a^2 \text{ such that }  \sigma(\alpha) = (x_1^2/ |\sigma_1(c)|, \dots, x_n^2/ |\sigma_n(c)|) \right \}.
\end{equation}
Note that this is a discrete subset of a union of ellipsoids in $\R^n$.
\begin{theorem}\label{thm:main-thm-ellipsoids}
Let $K$, $n$, and $V$ be as in Theorem \ref{thm:main-thm-spheres} and let $c$ and $\a$
be such that \eqref{eq:inverse-different-as-square} holds. Then, for every $\epsilon
\in \{ \pm 1 \}$ the subspace of all $f \in V$ satisfying $f(x)=0$ for all $x \in 
E(c, \a)$ and $\widehat{f} = \epsilon f$ is infinite dimensional.
\end{theorem}
The functions we produce for Theorems~\ref{thm:main-thm-spheres} and \ref{thm:main-thm-ellipsoids}  are quite explicit. The prototypical example is a 
linear combination of 16 Gaussians whose parameters $z  \in \H^n$ are of the form $z 
= \gamma \cdot \tau$ for a generic point $\tau \in \H^n$ and some special elements  $\gamma \in \PSL_2(\Oo_K)$, eight of which  are written down explicitly in the proof 
of Proposition~\ref{prop:nonzero-intersection-of-ideals}. The entries of the matrices can be computed if one knows some non-trivial units of~$\Oo_{K}$ in the congruence classes $1+4\Oo_K$ and $1+3\Oo_K$.

In the remaining parts of this introduction we give further explanations for  Theorems~\ref{thm:main-thm-spheres} and~\ref{thm:main-thm-ellipsoids} and add a few remarks. In \S\ref{subsec:general-lattices-in-intro}, we describe the other two results indicated in the abstract. 

\subsection{On the number of points in $\sqrt{\d^{-1}} \cap 
\sqrt{m}S^{n-1}$}\label{subsec:number-of-points} 
The cardinality of $\sqrt{\d^{-1}} \cap 
\sqrt{m}S^{n-1}$ is $2^n$ times the number of totally
non-negative elements in $\d^{-1}$ of trace $m \geq 0$. By
choosing a $\Z$-basis for $\Oo_{K}$ containing $1$ and considering the element
$\alpha_1 \in K$ such that the $\Q$-linear functional $y \mapsto \Tr(\alpha_1 y )$
takes the value $1$ on $y=1$ and zero on all other elements of the basis, we see
that $\Tr(\alpha_1) = 1$ and $\alpha_1 \in \d^{-1}$.
It follows that for all $m \in \Z$, we have
    \[\{ \alpha \in \d^{-1} \,:\, \Tr(\alpha) = m \} = m \alpha_1 + (\d^{-1})_0,  
    \qquad (\d^{-1})_0:=  \{ \alpha \in \d^{-1} \,:\, \Tr(\alpha) = 0 \}.\]
Thus, for $m \geq 0$, the subset of $\R^n$ whose cardinality we are interested in, can be written as
\[
\left( m \sigma(\alpha_1) + \sigma( (\d^{-1})_0) \right) \cap [0,\infty)^n
\]
whose cardinality equals that of
\[
\sigma( (\d^{-1})_0) \cap m( [0,\infty)^n - \sigma(\alpha_1)) 
\]
which is the set of lattice points of $\sigma((\d^{-1})_0) \subset \{ x  \in \R^n
\,:\, \sum_{i=1}^{n}{x_i} = 0 \}$ in a homogeneously expanding $(n-1)$-dimensional
region, allowing for an application of a standard estimate of the number of such
points, as $m \rightarrow \infty$. The necessary volume computations are done (for
any fractional ideal, in fact) in the work of Ash and Friedberg, see \cite[\S 5,
Prop. 5.1 and \S 6]{AS}. From the cited parts of their work, we deduce that
\begin{equation}\label{eq:asymptotic-formula-number-of-points}
\left| \sqrt{\d^{-1}} \cap \sqrt{m}S^{n-1} \right| = 2^n \frac{\sqrt{|\disc(K)|}}{(n-1)!}
m^{n-1} +O(m^{n-2}), \quad m \rightarrow \infty,
\end{equation}
where the implied constant may depend on $K$ and $n$. We
point out the following features of this asymptotic formula:
\begin{itemize}
\item The surface area of $\sqrt{m}S^{n-1}$ grows like $m^{\frac{n-1}{2}}$, so the points are more densely spaced than a constant number of points per 
unit surface area
on $S^{n-1}$.
\item We may increase the density of points by a constant factor, by taking the
discriminant of $K$ arbitrarily large, while keeping the degree $n$ fixed.
\item For small $m$, there may be no points in $\sqrt{\d^{-1}} \cap \sqrt{m}S^{n-1}$, 
but note that we can add \emph{any} finite set of points on these small spheres, by  
Remark~\ref{rem:finiteX}. 
\end{itemize}

\begin{figure}
	\includegraphics[width=0.45\textwidth]{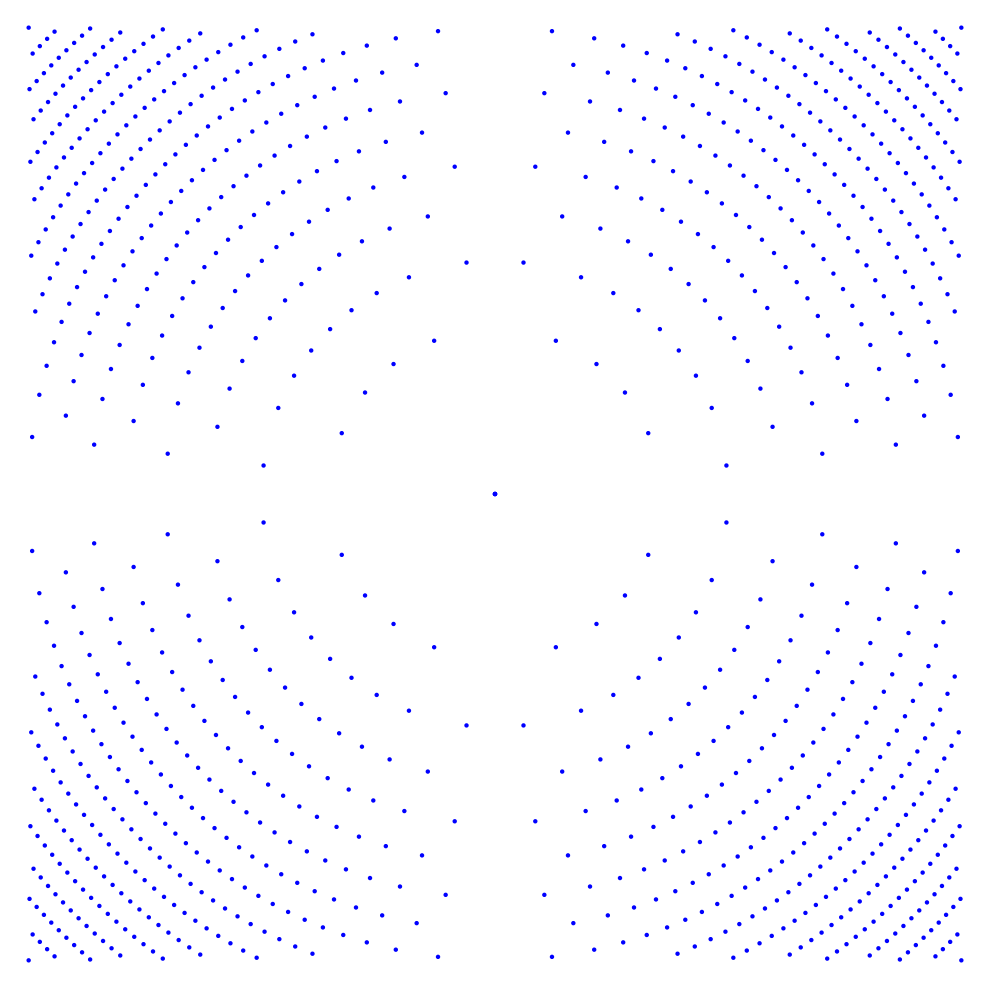}
   	\includegraphics[width=0.45\textwidth]{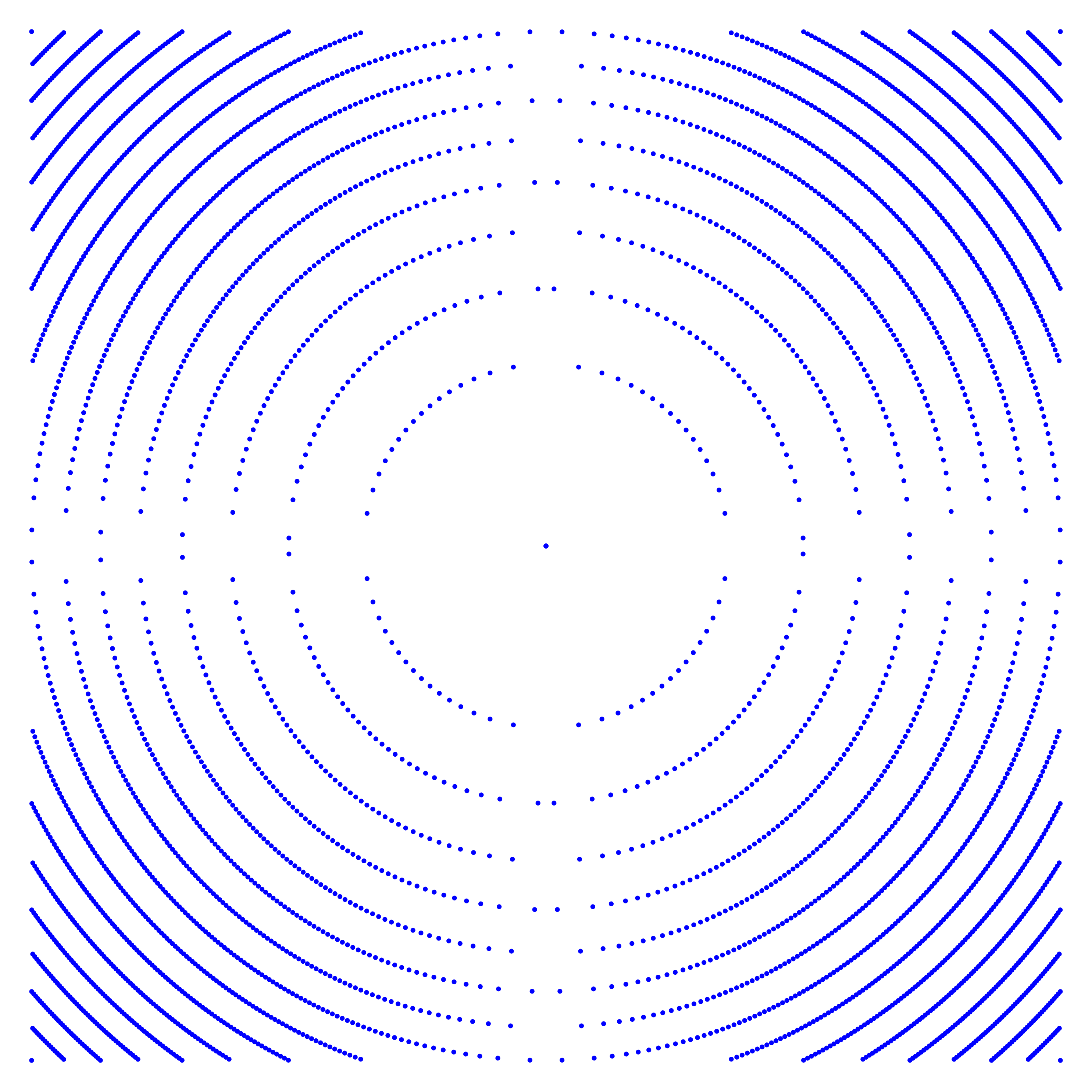}
	\caption{Non-uniqueness sets constructed from $\Q(\sqrt{17})$ 
    and $\Q(\sqrt{257})$}
    \label{fig1}
\end{figure}


\subsection{The relation between Theorem \ref{thm:main-thm-spheres} and Theorem \ref{thm:main-thm-ellipsoids}}\label{subsec:co-different-condition}
If the number $c$ in \eqref{eq:inverse-different-as-square}
can be taken totally positive, then $E(c, \a) = \sqrt{\Oo_K^{\vee}}$ and both theorems give the same result. Since we are free
to replace $c$ by $\varepsilon c$ for any unit $\varepsilon \in \Oo_{K}^{\times}$,
we can take $c$ totally positive, provided $K$ has units 
$\varepsilon$ of all possible sign patterns
$(\sigma_j(\varepsilon)/|\sigma_j(\varepsilon)|)_{1 \leq j \leq n} \in \{ \pm 1
\}^n$. In the real quadratic case, the latter is equivalent to the fundamental unit
having norm $-1$.  Such conditions are studied more generally in the literature, via the notion of signature rank.
 
In fact, whenever $K/ \Q$ is Galois and $n$ is odd, then $c = 1$ is admissible. In other words, the different is then exactly equal to the square of another ideal. This follows from Hilbert's formula, see \cite[Ex. 5.45, p. 253]{Mollin}.

Generally, recall that a large class of number fields which allows for an easy determination of admissible
$c$ and $\a$ in \eqref{eq:inverse-different-as-square} is given by monogenic number
fields. For example, for any irreducible monic polynoimal $P \in \Z[X]$  with
square-free discriminant and no complex roots, we can take $K= \Q(\alpha) \subset \R$ 
for some root~$\alpha$ of~$P$. Then it is well-known that $\Oo_{K} = \Z[\alpha]$ and 
$\Oo_{K}^{\vee} =
\frac{1}{P'(\alpha)} \Oo_{K}$, so that $(c, \a) = (1/P'(\alpha), \Oo_{K})$ is
admissible in \eqref{eq:inverse-different-as-square}. 

We note further that, if there is a constant $\theta>0$ so that $|\sigma_j(c)|=\theta$ for all
$j=1,\dots,n$, then  the set $E(c, \a)$ is contained in the union of spheres 
$\sqrt{\theta m}S^{n-1}$ (rather than in a union ellipsoids). This happens for some 
real-quadratic fields, see \S \ref{subsec:illustration-real-qaudratic}.

\subsection{Real quadratic fields}\label{subsec:illustration-real-qaudratic}
To illustrate the theorems in the case $n = 2$, consider a real quadratic field $K
=\Q(\sqrt{D})$ as a subfield of $\R$ of discriminant $D, \sqrt{D}>0$  and for $x \in K$ write
$\sigma_1(x) = x$ and $\sigma_2(x) =: x'$ so that $\sqrt{D}' = - \sqrt{D}$. Define $\omega :=(D+\sqrt{D})/2$ and $c := 1/ \sqrt{D}$. Then $\Oo_{K} = \Z + \Z \omega$ and $\Oo_{K}^{\vee}= c\Oo_{K} = c\Oo_{K}^2$ 
(square of a fractional ideal). Thus, every element of $\Oo_{K}^{\vee}$ may be written as $\alpha =
c(\ell + m \omega)$ for $\ell, m \in \Z$ and has $\Tr(\alpha) = \ell \Tr(c) + m
\Tr(\omega c) = m$. The element $\alpha$ is totally non-negative if and only if $m
\geq 0$ and $-m \omega \leq \ell \leq -m \omega'$. This shows that
\[
\left|\sqrt{\Oo_{K}^{\vee}} \cap \sqrt{m}S^1 \right| =2 \left|\Z \cap [-m \omega ,-m
 \omega']\right| \sim 2m \sqrt{D}, \quad m \rightarrow 
 \infty,
\]
which exemplifies \eqref{eq:asymptotic-formula-number-of-points} and Theorem \ref{thm:main-thm-spheres} in the simplest case.

Let us now illustrate Theorem \ref{thm:main-thm-ellipsoids} with $\a =
\Oo_{K}$ and the above value of $c$, which is \emph{not} totally positive and
satisfies $|\sigma_1(c)| = |\sigma_2(c)| = \frac{1}{ \sqrt{D}}$. We assume that $4|D$ and set $d:= D/4 \equiv 2,3 \pmod{4}$, so that $\Oo_K = \Z + \Z\sqrt{d}$. Then $E(c, \a)$ is the set of $x = (x_1, x_2) \in \R^2$ such that
\[
(x_1^2, x_2^2) =  \frac{1}{2\sqrt{d}} \left( a + b \sqrt{d}, a-b\sqrt{d}  \right)
\]
for some $a,b \in \Z$ satisfying $|b\sqrt{d}| \leq a$ and $x_1^2 + x_2^2 = 
\frac{a}{\sqrt{d}}$. In other words, $E(c, \a)$ is a discrete subset of a union of 
circles of radii $\sqrt{a/\sqrt{d}}$, for all integers $a \ge 0$ with about 
$a/\sqrt{d}$ 
many points on each. If $D \equiv 1 \pmod{ 4}$, then 
$E(c, \a)$ is a discrete subset of the union of all circles of radii 
$\sqrt{t/\sqrt{D}}$ for all integers $t \geq 0$ with about $2t/\sqrt{D}$ points on 
each. 
\subsection{A minor generalization of Theorem \ref{thm:main-thm-spheres} }\label{subsec:multi-radial-case}
The space of Gaussians defined in Theorem \ref{thm:main-thm-spheres} is a subspace of
the space of Schwartz functions on $\R^n$ that are even in each variable (and it
turns out to be dense in that space, see Proposition \ref{prop:density-of-Gaussians}). More generally, Theorem \ref{thm:main-thm-spheres} holds and will be proved in the following  setting.

Let $d,n \geq 1$ be integers and consider a partition $d = d_1 + \dots + d_n$ of $d$. 
We view the Euclidean space $\R^{d}$ as the product space $\R^{d_1} \times \dots 
\times \R^{d_n}$ and elements $x \in \R^d$ as $n$-tuples $x =
(x_1, \dots, x_n)$ where $x_j \in \R^{d_j}$. The group $\Or(d_1)
\times \dots \times \Or(d_n)$ embeds block-diagonally into the orthogonal 
group~$\Or(d)$. Denote by $\Ss(\R^d)^{\Or(d_1) \times \dots \times
\Or(d_n)}$ the space of Schwartz functions on $\R^d$ that are radial in each of 
the~$n$ variables $x_j \in \R^{d_j}$. Such functions can be identified with functions 
on $[0,+\infty)^n$ and we freely use this identification to evaluate them on 
$n$-tuples of non-negative real numbers. An $\Or(d_1) \times \dots \times
\Or(d_n)$-invariant function on $\R^d$ will be said to \emph{vanish on }$\sqrt{\Oo_K^{\vee}}$, if $f(x) = 0$ for all $x \in \R^d$ such that $(|x_1|^2, \dots, |x_n|^2) =  (\sigma_1(\alpha), \dots, \sigma_n(\alpha))$ for some $\alpha \in \Oo_K^{\vee}$.

Besides the case where all $d_j$ are equal,  our proof of Theorem 
\ref{thm:main-thm-ellipsoids} does not seem to easily generalize to the more general 
setting that we have just described, for technical reasons having to do with the 
existence of automorphic factors, see \S \ref{sec:pf-thm-2}.

\subsection{General lattices and a radial uniqueness result}\label{subsec:general-lattices-in-intro}
As already mentioned above, in \S \ref{sec:algebraic-obstructions} we will consider 
general lattices $ \Lambda \subset \R^n$ and their square roots  $\sqrt{\Lambda}:=  
\{(x_1, \dots, x_n) \in \R^n \,:\, (x_1^2, \dots, x_n^2) \in \Lambda \}$. In \S 
\ref{sec:set-up-with-lattices}, we will explain (mainly for motivational purposes) a 
natural equivalent formulation of a Fourier interpolation formula using the pair 
of sets $(\sqrt{\Lambda_1}, \sqrt{\Lambda_2})$ for lattices $\Lambda_1,\Lambda_2 
\subset 
\R^n$ in terms of generating series, viewed as functions on $\H^n$ and describe their 
modular transformation properties in terms of a certain subgroup $\Gamma(L_1, 
L_2) \leq \PSL_2(\R)^n$, where $L_i = 2 \Lambda_i^{\vee}$.  We will prove in Proposition 
\ref{prop:non-existence-of-good-lattices} that, for $n \geq 2$,  there is no pair of 
lattices $(L_1, L_2)$ such that the group $\Gamma(L_1, L_2)$ 
is discrete and at the same time the free inner product of two subgroups of upper- 
and lower triangular elements isomorphic to $L_1$ and $L_2$ respectively. We prove 
that the latter property of $\Gamma(L_1, L_2)$ is a necessary condition for the 
existence of such an interpolation formula (Proposition \ref{prop:necessity-of-F})  
and we argue why discreteness might be necessary as well.

From this more general point of view we return in  \S \ref{sec:interpolation-result} to the case $n  = 1$ and prove: 

\begin{theorem-with-no-label}[$=$ Theorem \ref{thm:interpolation-theorem-lambda-bigger-2} $+$ Corollary \ref{cor:non-radial-uniqueness-corollary-lambda-bigger-two} in \S \ref{sec:interpolation-result}]
For all $d \geq 5$ and all positive reals $\alpha, \beta$ such that $\alpha \beta  \geq 1$ the pair \begin{equation}\label{eq:dense-uniqueness-set-intro}
\left(\cup_{m \geq 1}{\sqrt{m/ \alpha}S^{d-1}}, \cup_{m \geq 1}{\sqrt{m/ \beta}S^{d-1}} \right)
\end{equation}
is a Fourier uniqueness pair for $\mathcal{S}(\R^d)$ and there exists a linear interpolation formula which proves this. Furthermore, if $\alpha \beta > 1$, then \eqref{eq:dense-uniqueness-set-intro} remains a uniqueness pair after removing any finite number of spheres from both sides.
\end{theorem-with-no-label}
The radial interpolation result (Theorem \ref{thm:interpolation-theorem-lambda-bigger-2}) will be proved via a series construction generalizing the one used in \cite{S} from $\Gamma(2)$ to the subgroup of $\PSL_2(\R)$ generated by \[
\begin{pmatrix}
1 & 2 \alpha \\ 0 & 1 
\end{pmatrix}, \qquad \begin{pmatrix}
1 & 0 \\ 2 \beta & 1 
\end{pmatrix}  .
\]
For $\alpha \beta \geq 1$, it is conjugate in $\PSL_2(\R)$  to a normal subgroup of 
index two in a Hecke group  and is isomorphic to 
$\Gamma(2)$. For $\alpha \beta >1$ these groups have infinite covolume and infinite dimensional spaces of modular forms. The latter fact was proved by Hecke \cite[\S 3]{H1} and his construction of linearly independent modular forms allows us to remove finitely many spheres from \eqref{eq:dense-uniqueness-set-intro}. 
\subsection{Some notation}\label{subsec:notation}
Besides the notation introduced above, we will also use the following general
notation throughout the paper. For $z \in \H$, the number $z/i$ belongs to the right
half plane $\mathcal{H}:= \{w \in  \C \,:\, \real(w) > 0 \}$ and on it, we always use
the branch of the logarithm $w \mapsto \log(w)$ that takes real values on $(0, +
\infty)$. For any $z \in \H$ and $k \in \C$, we thus define $(z/i)^{k} =
\exp(k \log(z/i))$. For $x \in \R$ we define $\sgn(x) \in \{-1, 0, 1 \}$ as $\sgn(x)
= x/|x|$ if $x \neq 0$ and $\sgn(0):=0$.

In the setting of \S \ref{subsec:multi-radial-case} we will work with complex 
Gaussians, parameterized by points $z = (z_1, \dots, z_n) \in \H^n$ and 
defined as
\begin{equation}\label{eq:def-gaussian}
g(z, x) = e^{\pi i  z_1 |x_1|^2} \cdots e^{\pi i  z_n |x_n|^2}, \quad x_j \in \R^{d_j}.
\end{equation}
We sometimes also view $g$ as a map $g: \H^n \rightarrow \Ss(\R^d)^{\Or(d_1) \times
\cdots \times \Or(d_n)}$, so that from this point of view $g(z)(x) = g(z,x)$. 
Moreover, we have, for all $z \in \H^n$,
\begin{equation}\label{eq:fourier-transform-gaussian}
\widehat{g(z)} = (z_1/i)^{-d_1/2} \cdots (z_n/i)^{-d_n/2} g(-1/z), \quad -1/z := 
(-1/z_1, \dots, -1/z_n).
\end{equation}
More specific notation will be introduced in the body of the paper. 

\subsection*{Acknowledgments}
The second author would like to thank Maryna Viazovska for sharing ideas and 
techniques during his previous work~\cite{S}, which were also useful in \S \ref{sec:interpolation-result}.

\section{Proof of Theorem \ref{thm:main-thm-ellipsoids}}
\label{sec:pf-thm-2}

The goal of this section is to prove Theorem~\ref{thm:main-thm-ellipsoids}. 
In \S\ref{sec:hilbert-modular-group-and-subgroups} we introduce some notation and 
define a ``theta-subgroup" $\Gamma_{\vartheta}$ of the Hilbert 
modular group~$\PSL_2(\Oo_K)$.
In \S\ref{sec:automorphic-factor} we define a slash action of the group algebra 
$\C[\Gamma_{\vartheta}]$ on complex-valued functions on a product of 
upper and lower half planes, via theta functions. The examples of non-trivial 
functions satisfying the vanishing conditions of 
Theorem~\ref{thm:main-thm-ellipsoids} will be given as Gaussians slashed 
with suitable elements in $\C[\Gamma_{\vartheta}]$. Lemmas  \ref{lem:action-fourier} 
and \ref{lem:difference-of-translates-of-gaussians}  will show that ``suitable'' 
means to belong to the intersection of two right ideals in $\C[\Gamma_{\vartheta}]$. 
In \S 
\ref{subsec:lemmas}, we will show that this intersection is infinite dimensional and 
conclude the   proof of Theorem \ref{thm:main-thm-ellipsoids} in \S 
\ref{sec:conclusions-pf-thm-2}. 

\subsection{Hilbert modular groups and subgroups}\label{sec:hilbert-modular-group-and-subgroups}
As in \S \ref{sec:introduction}, we consider a totally
real number field~$K$ of degree $n = [K: \Q] \geq 2$. As in
\eqref{eq:inverse-different-as-square}, we choose and fix $c \in K^{\times}$ and a
fractional ideal $\a \subset K$ so that $\d^{-1} = c \a^2$, where $\d$ is the different of $K$. Depending upon these
quantities we define signs $\delta_j := \sgn(\sigma_j(c))$, a vector of signs $\delta
= (\delta_j)_{1 \leq j \leq n} \in \{ \pm 1\}^n$ and
\[
\H_{\delta}^n:= \{z = (z_1, \dots, z_n) \in \C^n \,:\, \imag(\delta_j z_j) >0 \text{ for all } j \in \{1, \dots, n \} \}.
\]
Instead of the ones in~\eqref{eq:def-gaussian}, we will work, for all of \S\ref{sec:pf-thm-2}, 
with Gaussians
\begin{equation}\label{eq:signed-gaussian}
g_{\delta}(z) \in  \Ss(\R^n) \quad \text{ defined by } \quad g_{\delta}(z)(x) :=
g_{\delta}(z,x) := e^{ \pi i \sum_{j=1}^{n}{\delta_j z_j x_j^2}}, \quad  z \in
\H_{\delta}^n,\, x \in \R^n.
\end{equation}
We consider the Hilbert modular group $\Gamma := \PSL_2(\Oo_{K})$ and denote
    \[
    S = \begin{pmatrix}
    0 & -1\\
    1 & 0
    \end{pmatrix}, \qquad T^{\beta} = \begin{pmatrix}
    1 & \beta\\
    0 & 1
    \end{pmatrix}, \quad \beta \in \Oo_{K}, \qquad M(\varepsilon) = \begin{pmatrix}
    \varepsilon & 0\\
    0 & \varepsilon^{-1}
    \end{pmatrix}, \quad  \varepsilon \in \Oo_{K}^{\times}\,,
    \]
viewing these as elements of $\Gamma$. Next, we embed $\Gamma$ into
$\PSL_2(\R)^{n}$ via the real embeddings $\sigma_j$. The latter group and hence
$\Gamma$ itself, acts on $\H_{\delta}^n$ via fractional linear transformations. This
action is faithful and we sometimes identify a group element with the associated
automorphism of $\H_{\delta}^n$, in particular when writing compositions of maps. Define
\[
\Gamma_{\vartheta} := \left \langle \{ S \} \cup \{ T^{2\beta}\}_{\beta \in \Oo_{K}}
\cup \{ M(\varepsilon) \}_{\varepsilon \in \Oo_{K}^{\times}} \right \rangle \leq \Gamma.
\]
\begin{remark}\label{rmk:theta-group}
Let $\tilde{\Gamma}_{\vartheta}$ denote the image in $\Gamma$ of the
group of matrices in $\SL_2(\Oo_{K})$ which  reduce to $\begin{psmallmatrix} \ast &
0\\ 0 & \ast \end{psmallmatrix}$ or $ \begin{psmallmatrix} 0 & \ast\\ \ast & 0
\end{psmallmatrix}$ in $\SL_2(\Oo_{K} /2 \Oo_{K})$. By definition,
$\Gamma_{\vartheta}
\leq \tilde{\Gamma}_{\vartheta}$ and equality is known to hold (at least) in the case
$K = \Q(\sqrt{5})$ (see \cite[\S1]{M}). Even though it would be convenient, we do 
not need to know equality in general and only mention it to provide context (but we will also refer to this group in the proof of Proposition \ref{prop:nonzero-intersection-of-ideals}).  
\end{remark}
\subsection{Automorphic factors and slash action}\label{sec:automorphic-factor}
Our task  here is to define a suitable automorphic factor and a corresponding slash 
action of $\Gamma_{\vartheta}$ on spaces of functions on $\H_{\delta}^{n}$ so that 
the action of~$S$ matches with the Fourier transform acting on Gaussians and so 
that $T^{2\beta}$ simply acts as translation by~$2\sigma(\beta)$. 
We will use theta functions attached to fractional ideals in $K$. 
Essentially the same functions were already studied  by Hecke
\cite[\S 56]{H2}.

We define the function $\vartheta : \H_{\delta}^n
\rightarrow \C$ by the absolutely and normally convergent series
\[
\vartheta(z) := \vartheta(z_1, \dots, z_n) := \sum_{\alpha \in \mathfrak{a}}{e^{\pi i \sum_{j=1}^{n}{z_j \sigma_j(c \alpha^2)}}},
\]
where we recall that $\d^{-1} = c \a^2$. We next determine the transformation behavior of $\vartheta$ under the generators of
$\Gamma_{\vartheta}$. These are certainly not new, but we include their proofs to keep the presentation self-contained. First, since $\a$ is an $\Oo_{K}$-submodule of $K$, we have, for every
$\varepsilon \in \Oo_{K}^{\times}$, and every $z \in \H_{\delta}^n$
\[
\vartheta(M(\varepsilon)z) =\vartheta( \sigma_1(\varepsilon)^2 z_1, \dots, \sigma_n(\varepsilon)^2 z_n) = \vartheta(z).
\]
Next, $\vartheta(T^{2\beta} z) = \vartheta(z)$ for all $z \in \H_{\delta}^n$ and
all $\beta \in \Oo_{K}$, since for all $\alpha \in \mathfrak{a}$,
\[
\sum_{j=1}^{n}{(z_j +2 \sigma_j(\beta)) \sigma_j(c \alpha^2)} = \sum_{j=1}^{n}{z_j
\sigma_j(c \alpha^2)} + 2 \Tr_{K/ \Q}(\beta c \alpha^2)
\]
and the above trace is an integer. To study the effect of $\vartheta$ under $S$ note that, by definition, $\vartheta(z)$ is the sum over the lattice $\sigma(\a)$ of the Schwartz function $f_z = g_{\delta}(|\sigma_1(c)|z_1, \dots, |\sigma_n(c)|z_n)$ whose Fourier transform is
\[
\widehat{f_z}(\xi) = \prod_{j=1}^{n}{(\delta_j |\sigma_j(c)| z_j/ i)^{-1/2}e^{\pi i
\delta_j (-1/ (|\sigma_j(c)| z_j)) \xi_j^2}}  = |\Nr_{K/
\Q}{(c)}|^{-1/2}\prod_{j=1}^{n}{(\delta_j z_j/ i)^{-1/2}e^{\pi i (-1/ z_j) (1/
\sigma_j(c)) \xi_j^2}} .
\]
By applying Poisson summation to the function $f_z$ and the lattice $\sigma(\a) \subset \R^n$, we get
\begin{align*}
\vartheta(z) &=  \frac{1}{\covol{(\sigma(\a))}} \sum_{\lambda^{\ast} \in \sigma(\a)^{\vee}}{\widehat{f_z}( \lambda^{\ast})} \\
&= \frac{1}{|\Nr_{K/ \Q}{(c)}|^{1/2}\covol{(\sigma(\a))}} \prod_{j=1}^{n}{( \delta_j
z_j/ i)^{-1/2}} \sum_{\beta \in c \a }{e^{\pi i \sum_{j=1}^{n}{(-1/z_j)(1/
\sigma_j(c)) \sigma_j(\beta)^2}}},
\end{align*}
where we used that $\a^{\vee} = c \a$, which follows from multiplying the relation  
$c \a^2 = \d^{-1}$, by $\a^{-1}$ and using the general formula $\b^{\vee} = 
\mathfrak{d}^{-1} \b^{-1}$. Writing $\beta = c \alpha$  and summing over $\alpha \in 
\a$, the above computation proves \[
\vartheta(z) = ( \delta_1 z_1/ i)^{-1/2} \cdots ( \delta_n z_n/ i)^{-1/2}  \vartheta(Sz)
\]
provided that $|\Nr_{K/ \Q}{(c)}|\covol{(\sigma(\a))}^2 = 1$ holds. This in turn
follows again from the relation $c \a^2 = \mathfrak{d}^{-1}$, the general volume
formula \eqref{eq:formula-covolume-fractional-ideal}  and properties of the ideal
norm.


We now define $\Omega_{\delta}^n := \{ z \in \H_{\delta}^n \,:\, \vartheta(z) \neq 0 
\}$, a nonempty open subset of $\H_{\delta}^n$ containing the product of  the imaginary 
axes, which is invariant under $\Gamma_{\vartheta}$ and the $1$-cocycle  $j_{\vartheta}: 
\Gamma_{\vartheta} \rightarrow \Hol{( \Omega_{\delta}^n, \C^{\times})}$ by
\begin{equation}\label{eq:def-automorphic-factor}
j_{\vartheta}(\gamma)(z) :=  j_{\vartheta}(\gamma, z) := \frac{\vartheta(\gamma z)}{\vartheta(z)}.
\end{equation}
Here, $\Hol{( \Omega_{\delta}^n, \C^{\times})}$ denotes the abelian group of all
nowhere vanishing, holomorphic functions on $\Omega_{\delta}^n$. Our computations
from above and the definitions imply that, for all $\beta \in \Oo_{K}$, all
$\varepsilon \in \Oo_{K}^{\times}$, all $z \in \Omega_{\delta}^n$ and all $\gamma_1,
\gamma_2 \in \Gamma_{\vartheta}$,
\begin{equation}\label{eq:properties-cocycle}
j_{\vartheta}(T^{2\beta})  =1, \quad j_{\vartheta}(M(\varepsilon)) = 1, \quad
j_{\vartheta}(S, z) = \prod_{j=1}^{n}{( \delta_j z_j/i)^{1/2}}, \quad
j_{\vartheta}(\gamma_1 \gamma_2) = (j_{\vartheta}(\gamma_1) \circ \gamma_2) \cdot
j_{\vartheta}(\gamma_1).
\end{equation}
It is not strictly necessary for our purposes, but, for convenience, we will lift 
$j_{\vartheta}$ to a cocycle $j_{\vartheta}: \Gamma_{\vartheta} \rightarrow 
\Hol{(\H_{\delta}^n, \C^{\times})}$. To explain how, note that, by our definition of 
$\Gamma_{\vartheta}$ via generators, and by \eqref{eq:properties-cocycle}, each 
function $j_{\vartheta}(\gamma)$ can we written as a finite product of functions 
$j_{\vartheta}(S) \circ\gamma'$ over some 
$\gamma'\in\Gamma_{\vartheta}$ and all of these are everywhere defined, holomorphic and nowhere vanishing on $\H_{\delta}^n$. Thus, we can (re-)define $j_{\vartheta}$ on \emph{generators}  by requiring that  \eqref{eq:properties-cocycle} holds. Any relation in $\Gamma_{\vartheta}$ will be respected in $\Hol{(\H_{\delta}^n, \C^{\times})}$ since the functions expressing the relation must agree on the non-empty open subset $\Omega_{\delta}^n \subset \H_{\delta}^n$. 

Finally, for any function $f$ on $\H_{\delta}^n$ with values in a complex vector space and any $\gamma \in \Gamma_{\vartheta}$, we define a new function $f|\gamma$ on $\H_{\delta}^n$ by
\begin{equation}\label{eq:definition-slash-action}
f| \gamma := j_{\vartheta}(\gamma)^{-1}  \cdot (f \circ \gamma), \quad \text{that is } \quad (f| \gamma)(z) = j_{\vartheta}(\gamma, z)^{-1}f (\gamma \cdot z).
\end{equation}
We extend this group action to the group algebra $\Rr := \C[\Gamma_{\vartheta}]$  in the usual way.


The next two lemmas hint at the usefulness of the action we just introduced, for the 
proof of Theorem \ref{thm:main-thm-ellipsoids}. Indeed, these Lemmas will essentially 
reduce the proof of Theorem \ref{thm:main-thm-ellipsoids} to  a purely algebraic 
statement about a right ideal in the algebra $\Rr$, which will be addressed in the 
next section.

\begin{lemma}\label{lem:action-fourier}
For every $A \in \Rr$ and $z \in \H_{\delta}^n$ we have $\mathcal{F}_{\R^n}((g_{\delta}|A)(z)) = (g_{\delta}|SA)(z)$.
\end{lemma}
\begin{proof}
By linearity, we may assume that $A \in \Gamma_{\vartheta}$. Given that $\widehat{g_{\delta}(z)} = j_{\vartheta}(S,z)^{-1}g_{\delta}(Sz)$ and the properties \eqref{eq:properties-cocycle},
\begin{align*}
\mathcal{F}((g_{\delta}|A)(z)) &= j_{\vartheta}(A, z)^{-1} \mathcal{F}(g_{\delta}(Az)) =j_{\vartheta}(A, z)^{-1}j_{\vartheta}(S, Az)^{-1} g_{\delta}(S(Az)) \\
&= j_{\vartheta}(SA,z)^{-1}g_{\delta}(SA z) = (g_{\delta}|SA)(z),
\end{align*}
as claimed.
\end{proof}
We denote by  $\Ii = \sum_{\beta \in \Oo_{K}}{ (1-T^{2\beta}) \Rr } $ the right ideal
generated by all elements $(1-T^{2\beta})$, $\beta \in \Oo_{K}$.
\begin{lemma}\label{lem:difference-of-translates-of-gaussians}
For all $A \in \Ii$ and all $z \in \H_{\delta}^{n}$, the function $(g_{\delta}|A)(z) : \R^n \rightarrow \C$ vanishes at all points $x = (x_1, \dots, x_n) \in \R^n$ for which there is $\alpha \in \a^{2}$ such that $x_j^2 = |\sigma_j(c)|\sigma_j(\alpha)$ for all $j$, that is to say, at all points of the set $E(c, \a)$, defined in \eqref{eq:def-set-E_F}.
\end{lemma}
\begin{proof}
By linearity, may assume that $A = (T^{2\beta}-1)\gamma$ for some $\gamma \in
\Gamma_{\vartheta}$ and some $\beta \in \Oo_{K}$. By definition and by
\eqref{eq:properties-cocycle}, we have
\[
(g_{\delta}|(T^{2\beta}-1)\gamma)(z) = (g_{\delta}|T^{2\beta}\gamma)(z)-(g_{\delta}|\gamma)(z) =
j_{\vartheta}(\gamma, z)^{-1}(g_{\delta}(\gamma z + 2\sigma(\beta))- g_{\delta}(\gamma z)).
\]
Set $\tau := \gamma z$. Then, for all $x = (x_1, \dots, x_n) \in \R^n$,
\[
g_{\delta}(\tau  + 2 \sigma(\beta))(x)- g_{\delta}(\tau)(x) = \left( e^{2 \pi i \sum_{j=1}^{n}{\delta_j\sigma_j(\beta) x_j^2}} -1 \right)g_{\delta}(\tau)(x).
\]
If there is $\alpha \in \a^2$ so that $x_j^2 = |\sigma_j(c)|\sigma_j(\alpha)$ for all $j$, then, since $\delta_j = \sigma_j(c)/|\sigma_j(c)|$, we have
\[
\sum_{j=1}^{n}{\delta_j\sigma_j(\beta) x_j^2 } = \sum_{j=1}^{n}{\sigma_j(c)
\sigma_j(\beta) \sigma_j(\alpha) } = \Tr_{K/ \Q}(c \beta \alpha) \in \Z,
\]
because $c \alpha \in \Oo_{K}^{\vee}$ and $\beta \in \Oo_{K}$. This proves what we want.
\end{proof}

\subsection{Ideals in the group algebra $\Rr = \C[\Gamma_{\vartheta}]$}\label{subsec:lemmas}
Lemma \ref{lem:action-fourier} and Lemma \ref{lem:difference-of-translates-of-gaussians} together show that, for any element 
$A \in \Rr$ which belongs to the ideal $\Ii$ and which  can also be written as $A = 
(1+ \epsilon S)A_1$ for some $A_1 \in \Rr$ and $\epsilon \in \{ \pm 1 \}$ is such 
that, for any $z \in \H_{\delta}^n$, the Schwartz function $f 
=(g_{\delta}|A)(z)$  vanishes at all points of the set $E(c, \a)$ and has Fourier 
transform $\widehat{f} = \epsilon f$. The next proposition will show that there are 
plenty of such elements $A$. It lies at the heart of our proof of Theorem  
\ref{thm:main-thm-ellipsoids} (and Theorem \ref{thm:main-thm-spheres}).
\begin{proposition}\label{prop:nonzero-intersection-of-ideals}
We have $(1-S) \Rr \cap \Ii \neq 0$ and $(1+S) \Rr \cap \Ii \neq 0$. Moreover, these 
intersections are infinite dimensional vector spaces over $\C$.
\end{proposition}
\begin{proof}
We first note that if $\Jj \subset \Rr$ is any \emph{nonzero} right ideal, then, 
since the group $\Gamma_{\vartheta}$ is infinite, we can produce an arbitrarily high 
number of right translates of a  single nonzero element in $\Jj$ that have disjoint 
supports (say), showing that $\dim_{\C}(\Jj) = \infty$. So we only need to show that 
$(1 \pm S) \Rr \cap \Ii \neq 0$.

To do that, we note that if two elements $ \gamma_1, \gamma_2 \in \Gamma_{\vartheta}$ 
have the same bottom row (possibly up to sign), 
then $\gamma_1- \gamma_2 = (1- \gamma_{2} \gamma_1^{-1}) 
\gamma_1 \in \Ii$. It thus suffices to construct  $A_{+ }, A_{-} \in \Rr$ such that 
$(1 -S) A_{-}$ and $(1 + S) A_{+}$ can be written  as non-trivial finite sums of 
differences of group elements with equal bottom row. We also know that 
left multiplication by~$S$ interchanges the rows of a matrix and switches the sign on 
the top. Guided by these two observations, we make the Ansatz
\[
A_{-} = \sum_{r \in \Z/2n \Z}{\gamma_r}, \quad \gamma_{r} = \begin{pmatrix}
c_{r-1} & d_{r-1}\\ c_{r} & d_{r}
\end{pmatrix}, \qquad A_{+} = \sum_{r \in \Z/2n \Z}{(-1)^r\gamma_r'}, \quad 
\gamma_{r}' =\begin{pmatrix}
 c_{r-1}' & d_{r-1}'\\ c_{r}' & d_{r}'
\end{pmatrix},
\]
where $n \geq 1$ and $c_r, d_r, c_r', d_r' \in \Oo_{K}$  are to be found so that all
elements $\gamma_r, \gamma_r'$ belong to $\Gamma_{\vartheta}$ and such that $0 \neq (1 \pm S)A_{ \pm }$ because these elements always belong to $\Ii$. Some experimentation shows that there are no non-trivial examples for $n =1, 2,3$  and further experimentation yields an example for $n = 4$ as follows. Choose $a,b,x,y \in \Oo_{K}$ such that
\begin{equation}\label{eq:key-equation}
(1+4a)(1+4x) = 1=(1-3b)(1-3y), \qquad axby \neq 0.
\end{equation}
This is possible by Dirichlet's unit Theorem, which implies that for all non-zero integral ideals $\a \subset \Oo_K$, the kernel of the natural map $\Oo_{K}^{\times} \rightarrow (\Oo_{K}/\mathfrak{a})^{\times}$ is infinite (use this for $\mathfrak{a} =
4 \Oo_{K}$ or $3 \Oo_{K}$). Consider then the elements $\gamma_r = \gamma_r'$ defined by
\begin{align*}
\gamma_0 = \begin{pmatrix}
1 & 0\\
0 & 1
\end{pmatrix}, \quad
\gamma_1 = \begin{pmatrix}
0 &1\\
-1 & 2a
\end{pmatrix}, \quad \gamma_2 = \begin{pmatrix}
-1 & 2a\\
2 & -(1+4a)
\end{pmatrix}, \quad \gamma_3 = \begin{pmatrix}
2 & -(1+4a)\\
\frac{1-4b}{1+4a} & 2b
\end{pmatrix},\\
\gamma_4 = \begin{pmatrix}
\frac{1-4b}{1+4a} & 2b\\
2 y & \frac{1-4y}{1+4x}
\end{pmatrix}, \quad \gamma_5 = \begin{pmatrix}
2 y & \frac{1-4y}{1+4x}\\
-(1+4x) & 2
\end{pmatrix}, \quad \gamma_6 = \begin{pmatrix}
-(1+4x) & 2\\
2x & -1
\end{pmatrix}, \quad \gamma_7 = \begin{pmatrix}
2x & -1\\
1 & 0
\end{pmatrix}.
\end{align*}
We claim that:  (i) each $\gamma_r$  belongs to $\Gamma_{\vartheta}$ and (ii) 
that $(1 \pm S)A_{ \pm } \neq 0$.  To prove (i), we first verify, by computing 
determinants and using \eqref{eq:key-equation}, that each $\gamma_r$ belongs 
to the congruence group $\tilde{\Gamma}_{\vartheta} \supset \Gamma_{\vartheta}$ 
defined in Remark \ref{rmk:theta-group}. On the other hand, for $r \neq 4$, 
either one of the diagonal or off-diagonal entries of $\gamma_r$ is a unit, 
so that,  by multiplying $\gamma_r$ from the right or the left by  
$S^{\delta_1}T^{2\alpha} S^{\delta_2}$ with suitable $\alpha \in \Oo_K$, 
$\delta_1, \delta_2 \in \{0,1 \}$, we obtain a matrix in $\tilde{\Gamma}_{\vartheta}$ 
one of whose diagonal or off-diagonal entries is zero and hence belongs to  
$\Gamma_{\vartheta}$. For $\gamma_4$, note that $\gamma_4 T^{2(1 + 4a)}$ has lower  right entry equal to $1+4a$, which is a unit.  


To verify  (ii) note that, since none of $a,b,x,y$ is zero, we have
$\{\gamma_r \}_{r \in \Z/8 \Z} \cap \{S,1\} = \{1 \}$, so that the coefficient of $1 \in \Gamma_{\vartheta}$ in the finite sum $(1   \pm S)A_{ \pm }$ is $1 \in \C$. 
\end{proof}

%
Having proved Proposition \ref{prop:nonzero-intersection-of-ideals} it remains to 
show that we can produce any number of linearly independent functions 
$(g_{\delta}|A)(z)$ by varying $A \in \Ii \cap (1 \pm S) \Rr$ and $z \in 
\H_{\delta}^n$ suitably. This will be achieved via the next  lemma and its consequences.
\begin{lemma}\label{lem:linear-independence}
Let  $c_1, \dots, c_{m} \in \C^n$ be pairwise distinct. Then the functions $g_{\mu}: \R^n \rightarrow \C$, $g_{\mu}(r) = e^{\sum_{j=1}^{n}{c_{\mu,j}r_j^2}}$, $\mu = 1, \dots, m$, are $\C$-linearly independent.
\end{lemma}
\begin{proof}
We induct on $m \geq 1$, the case $m =1$ being clear. If $m \geq 2$ and $\sum_{\mu=1}^{m}{t_{\mu}g_{\mu}} = 0$ for some $t_{\mu} \in \C$, we divide by $g_{1}$ and differentiate with respect to $r_j$, giving $\sum_{\mu=2}^{m}{2(c_{\mu,j}-c_{1,j})r_j t_{\mu }g_{\mu}(r)} =0$. By continuity we may also divide by $r_j$ and apply the inductive hypothesis to deduce $(c_{\mu,j}-c_{1,j})t_{\mu} = 0$ for all $j$ and $\mu$.  Since $c_1 \neq c_{\mu}$ for all $\mu \geq 2$, this implies $t_{\mu} = 0$ for all $\mu \geq 2$ and then also $t_{1} = 0$.
\end{proof}
\begin{corollary}\label{cor:injectivity-on-good-points}
Let $z \in \H_{\delta}^n$ be a point such that for all $\gamma, \omega \in
\Gamma_{\vartheta}$, we have $\gamma \neq \omega \Rightarrow \gamma z \neq \omega z$.
Then the map $\Phi_z :\Rr \rightarrow \Ss(\R^n)$, $\Phi_z(A) = (g_{\delta}|A)(z)$ is
injective.
\end{corollary}
\begin{proof}
Suppose that $A  = \sum_{\gamma \in \Gamma_{\vartheta}}{\lambda_{\gamma} \gamma} \in \Rr$  is such that $\Phi_z(A) =0$. Let $\{\gamma_1, \dots, \gamma_m\} = \{ \gamma \in \Gamma_{\vartheta} \,:\, \lambda_{\gamma} \neq 0 \}$ be the support of $A$ (with pairwise distinct $\gamma_i$). By assumption, we have $0=\Phi_z(A) = \sum_{i=1}^{m}{\lambda_{\gamma_i}j_{\vartheta}(\gamma_i,z)^{-1}g(\gamma_i z)}$, so $\lambda_{\gamma_i} = 0$ follows by applying Lemma \ref{lem:linear-independence} to $c_{\mu} = \pi i \gamma_{\mu} z$.
\end{proof}
There are uncountably many points $z \in \H_{\delta}^n$ satisfying the assumption of Corollary \ref{cor:injectivity-on-good-points}; let us call such points \emph{good} (for the field $K$). To see this, note that the set of good points contains (since $\Gamma_{\vartheta} \subset \Gamma$)
\begin{equation}\label{eq:good-intersection}
\bigcap_{1 \neq \gamma \in \Gamma}{\{ z \in \H_{\delta}^n \,:\, \gamma z \neq z \}}= \H_{\delta}^n \sm  \bigcup_{1 \neq \gamma \in \Gamma}{\{ z \in \H_{\delta}^n \,:\,  \gamma z = z \}}
\end{equation}
and that each fix point set in the union on the right is either empty or a singleton 
set. Since $\Gamma$ is countable and $\H_{\delta}^n$ is uncountable, the above set is 
indeed uncountable. (It is  moreover dense in $\H_{\delta}^n$, by Baire's theorem, 
but we won't need this fact.) We call a point belonging to the intersection 
\eqref{eq:good-intersection} a \emph{generic} point (for the field $K$). Thus, all 
generic points are good.

\subsection{Conclusion}\label{sec:conclusions-pf-thm-2}
We can now give the proof of Theorem  \ref{thm:main-thm-ellipsoids}.
\begin{proof}[Proof of Theorem  \ref{thm:main-thm-ellipsoids}]
Fix $\epsilon \in \{ \pm 1\}$ and a good point $z \in 
\H_{\delta}^n$ for the field $K$.  By Corollary \ref{cor:injectivity-on-good-points}, the linear
map $\Phi_z :\Rr \rightarrow \Ss(\R^n)$, $\Phi_z(A) = (g_{\delta}|A)(z)$ is injective. Note that it takes values in the space $V \subset \Ss(\R^n)$ of all linear combinations of Gaussians. By Proposition \ref{prop:nonzero-intersection-of-ideals}, the space
$\Jj_{\epsilon}: = \Ii \cap (1 +
\epsilon S)\Rr$ and hence also $\Phi_z(\Jj_{\epsilon})$ is infinite dimensional. On the
other hand, by Lemma \ref{lem:action-fourier} and Lemma
\ref{lem:difference-of-translates-of-gaussians}, the space $\Phi_z(\Jj_{\epsilon})$ is
contained in the space of all $f \in V$ satisfying $\widehat{f} = \epsilon f$
and $f(x) = 0$ for all $x \in E(c, \a)$, proving the Theorem.
\end{proof}

\section{Proof of Theorem \ref{thm:main-thm-spheres}}\label{sec:pf-of-thm-1}
In this section we give the proof of Theorem \ref{thm:main-thm-spheres}. We let, as
usual,~$K$ be a totally real number field of degree $n \geq 2$ and use notation
associated with it as in~\S\ref{sec:introduction}. We will also use some of the notation and results of \S\ref{sec:pf-thm-2}, in
particular, the eight elements $\gamma_r$, $r \in \Z/8\Z$ given in the proof of
Proposition \ref{prop:nonzero-intersection-of-ideals}, Lemma
\ref{lem:linear-independence} and the notion of a \emph{generic} point for $K$, as defined near \eqref{eq:good-intersection} (but with $\H_{\delta}^n$ replaced by $\H^n$). The entries of the
matrices $\gamma_r$ depend on a non-trivial solution  $a,b,x,y \in \Oo_{K}$ to the equation
\eqref{eq:key-equation}. Further below, we will need to assume in addition that
\begin{equation}\label{eq:positivity-assumption}
\text{\emph{all four units} } \quad (1+4a), (1+4x), (1-3b), (1-3y) \in
\Oo_{K}^{\times} \quad \text{\emph{are totally positive}}.
\end{equation}
This is possible, since the subgroup  of totally
positive units in $\Oo_{K}^{\times}$ is infinite (indeed, already the subgroup of 
squared units is infinite, by Dirichlet's unit Theorem).

As advertised in \S \ref{subsec:multi-radial-case},
we work for all of \S \ref{sec:pf-of-thm-1}, on $\R^d
= \R^{d_1} \times \cdots  \times \R^{d_n}$ and with the corresponding Gaussians $g(z)
: \R^d \rightarrow \C$, defined as in \eqref{eq:def-gaussian}. This will prove a more 
general statement than Theorem~\ref{thm:main-thm-spheres}.  We also fix a sign $\epsilon \in
\{ \pm 1\}$ and consider a generic point $z = (z_1, \dots, z_n) \in \H^n$. We  use the short hand notation $\mu(z) := \prod_{j=1}^{n}{(z_j/i)^{d_j/2}} \in \C^{\times}$.  For a set of coefficients
$\{\lambda_r(z)\}_{r \in \Z/8 \Z} \subset \C$ which we will determine later,
consider the linear combination of Gaussians
    \[    h_{z} = \sum_{r \in \Z/8 \Z}{\lambda_r(z) g(\gamma_r  z)},    \]
where the matrices $ \gamma_r \in \Gamma$, $r \in \Z/8 \Z$, are as in the proof of Proposition \ref{prop:nonzero-intersection-of-ideals}. We define and compute, using \eqref{eq:fourier-transform-gaussian},
\begin{align*}
f_z := h_z + \epsilon \widehat{h_z} &=  \sum_{r \in \Z/8 \Z}{\lambda_{r}(z) g(\gamma_{r}  z)} +  \sum_{r \in \Z/8 \Z}{ \epsilon \lambda_r(z) \mu(\gamma_r z)^{-1} g(S\gamma_r  z)}  \\
&=  \sum_{r \in \Z/8 \Z}{\left(\lambda_{r-1}(z) g(\gamma_{r-1}  z) + \epsilon \lambda_r(z) \mu(\gamma_r z)^{-1} g(S\gamma_r  z)\right) }.
\end{align*}
By construction, $\widehat{f_z} = \epsilon f_z$. We claim that the coefficients 
$\lambda_r(z)$ can be chosen in such a way that
\begin{equation}\label{eq:recursion-condition-lambda_r}
\lambda_r(z) \neq 0 \quad \text{and} \quad 
\epsilon \lambda_r(z) \mu(\gamma_r z)^{-1} =- \lambda_{r-1}(z) \quad \text{ for all }  r \in \Z/ 8 \Z. 
\end{equation}
We postpone the proof of the claim to a later stage. Assuming it, we get
\begin{equation}\label{eq:formula-f_z}
f_z = \sum_{r \in \Z/8 \Z}{\lambda_{r-1}(z) \left( g(\gamma_{r-1} z)- g(S \gamma_r z) \right)}.
\end{equation}
Each difference $g(\gamma_{r-1} z)- g(S \gamma_r z)$ vanishes (in the sense defined in \S \ref{subsec:multi-radial-case}) on $\sqrt{\Oo_{K}^{\vee}}$, since, by construction, $S \gamma_r = T^{2 \beta_r}
\gamma_{r-1}$ for some $\beta_r \in \Oo_{K}$ and so
\[
S\gamma_r z =T^{2\beta_r} \gamma_{r-1}z=   \gamma_{r-1}z+ 2 \sigma(\beta_r),
\]
implying  that, if there is $\alpha \in \Oo_{K}^{\vee}$ so that  $|x_j|^2 = \sigma_j(\alpha)$ for all $j$, then
\[
f_z(x) = \sum_{r \in \Z/8 \Z}{\lambda_{r-1}(z) e^{\pi i \sum_{j=1}^n{\sigma_j(\gamma_{r-1})z_j|x_j|^2}} \left(1- e^{ 2 \pi i \sum_{j=1}^n{\sigma_j(\beta_r) \sigma_j(\alpha)}} \right)} =0,
\]
because $\sum_{j=1}^{n}{\sigma_j(\beta_r) \sigma_j(\alpha)} = \Tr(\alpha \beta_r) \in\Z$. 

So far, $z$ was an arbitrary generic point. We now verify that $f_z \neq 0$ and that 
we can produce an arbitrary number of linearly independent functions of this form. 
Since $z$ is generic for $K$, we have
\[
\{ r \in \Z/8 \Z \,:\, \gamma_r z = z \text{ or } S \gamma_{r}z = z\} =  \{ r\in \Z/8 \Z \,:\, \gamma_r =1 \text{ or } S \gamma_r = 1 \} = \{0 \}
\]
and this shows $f_z \neq 0$ via Lemma \ref{lem:linear-independence} and $\lambda_r(z) \neq 0$ for all $r$. Assume we have constructed linearly independent $f_{\tau_1}, \dots, f_{\tau_m}$ of this form with generic $\tau_j \in \H^n$ (here, the subscripts do \emph{not} denote coordinates). Since the set of generic points for $K$ is infinte (indeed uncountable), we can choose a generic point $\tau_{m+1} \in \H^n \smallsetminus \{ \tau_1, \dots, \tau_m \}$  and the functions $f_{\tau_1}, \dots, f_{\tau_{m+1}}$ are then linearly independent as well. Indeed, if $0 = \sum_{i=1}^{m+1}{t_i f_{\tau_i}} = \sum_{w \in \H^n}{a_w g(w)}$, for $t_i \in \C$ and (unique) $a_w \in \C$ we find that $0 =a_{\tau_i} = t_{i}$ for all $i$, as desired.

To finish the proof of Theorem  \ref{thm:main-thm-spheres}, it remains to  prove the claim made in \eqref{eq:recursion-condition-lambda_r}. A short calculation shows that this claim is 
equivalent to
\begin{equation}\label{eq:consistency-product}
1=\prod_{r \in \Z/8 \Z}{\mu(\gamma_r z)} = \prod_{r \in \Z/8 \Z}{\prod_{j=1}^{n}{((\sigma_j(\gamma_r)z_j)/i)^{d_j/2}}}.
\end{equation}
Indeed, if \eqref{eq:consistency-product} holds, we can choose an arbitrary constant $\lambda_0 = \lambda_0(z) \in \C^{\times}$ and put  
\[
\lambda_{k + 8 \Z}(z) = (-\epsilon)^k \lambda_0 \prod_{1 \leq i \leq k}{\mu(\gamma_{i + 8 \Z} z)} \quad \text{for } 1 \leq k \leq 7.
\]
Let us denote the product on the right of \eqref{eq:consistency-product} by 
$\rho(z)$. From the specific shape of the $\gamma_r$, it is clear that $\rho(z)^8 = 
1$. Since $\H^n$ is connected, we deduce that the continuous function $z \mapsto 
\rho(z)$ is constant, with constant value given by an eighth root of unity $\rho$. To 
determine $\rho$, we will take the points $z_j$ to $i \infty$.  For this we need the 
following lemma.
\begin{lemma}\label{lem:argument-formula}
For any $g = \begin{psmallmatrix}
a & b\\
 c & d
\end{psmallmatrix} \in \SL_2(\R)$, we have
\begin{equation}\label{eq:argument-formula}
\lim_{y \rightarrow \infty}{  \frac{((g  \cdot (iy))/i)^{1/2}}{|g \cdot (iy)|^{1/2}}} = \exp(- \tfrac{\pi i}{4} \sgn( a c)) = e(-\tfrac{1}{8}\sgn( a c )),
\end{equation}
where we write $e(w) = \exp(2 \pi i w)$ and where, here and elsewhere, the conventions of \S \ref{subsec:notation} are in place.
\end{lemma}
We defer the proof  of Lemma \ref{lem:argument-formula} to the end of this section. Writing
\[
\rho =\frac{\rho}{|\rho|}= \lim_{y \rightarrow \infty}{ \frac{\rho((iy, \dots, iy))}{|\rho((iy, \dots, iy))|}}
\]
and applying formula \eqref{eq:argument-formula} (and using the fact that the $d_j$ are 
integers\footnote{We arrived at a minor conflict of notation: There are dimensions $d_j \in \N$, $j \in \{1, \dots, n \}$ and elements $d_{r} \in \Oo_K$, $r \in \Z/8 \Z$, the entries of the right columns of the elements $\gamma_r$. The $d_r\in \Oo_K$ won't play a role in the remaining argument.})
, we see that
\begin{equation}\label{eq:formula-for-rho}
\rho = \prod_{r \in \Z/8 \Z}{\prod_{j=1}^{n}{e \left(-\tfrac{d_j}{8}\sgn(\sigma_j(c_r c_{r-1})) \right)}},
\end{equation}
where we recall that $c_r$ denotes the lower left entry of $\gamma_r$ and $c_{r-1}$ the upper left entry of $\gamma_r$. Let us write down the eight products $c_{r} c_{r-1}$ appearing in \eqref{eq:formula-for-rho}. For $\alpha_1, \alpha_2 \in K$, we write $\alpha_1 \equiv \alpha_2$ to express that there is a totally positive $\beta \in K^{\times}$ so that $\alpha_2= \alpha_1 \beta$. Then, by assumption \eqref{eq:positivity-assumption},
\begin{align*}
&c_0 c_{7} = 0  &&c_1 c_0 = 0 \\
&c_2 c_1 =-2 \equiv -1 &&c_3 c_2 =  2\frac{1-4b}{1+4a}  \equiv 1-4b   \\
&c_4 c_3 =\frac{1-4b}{1+4a} (2y)  \equiv (1-4b)y &&c_5 c_4 =  -(2 y)(1 +4x) \equiv -y \\
&c_6 c_5 =-2x (1 +4x)\equiv -x &&c_7 c_6 = 2x \equiv x.
\end{align*}
 We introduce the short hands
\[
\eta_j := \sgn(1-4 \sigma_j(b)) , \qquad \xi_j = \sgn(\sigma_j(y)).
\]
Interchanging the order of multiplication in \eqref{eq:formula-for-rho},  using the above list of identities and noting that $c_0 c_7, c_1 c_0$ don't contribute, while the contributions of  $c_6 c_5$ and $c_7 c_6$ cancel, we arrive at the formula
\[
\rho= e(-\tfrac{1}{8} \Sigma), \quad \text{where } \quad  \Sigma= \sum_{j=1}^n{d_j(-1 + \eta_j + \xi_j \eta_j -\xi_j)}= \sum_{j=1}^n{d_j(\eta_j-1)(\xi_j + 1)}.
\]
We claim that for each $j$ we have $(\eta_j-1)(\xi_j + 1) = 0$, or equivalently
\begin{equation}\label{eq:either-or-inequality}
1-4 \sigma_j(b)>0 \qquad \text{ or } \qquad \sigma_j(y) < 0.
\end{equation}
By \eqref{eq:key-equation}, we have $(1-3b)(1-3y)=1$ and hence
\[
(1-3\sigma_j(b))(1-3\sigma_j(y))=1.
\]
By assumption \eqref{eq:positivity-assumption}, both factors in this product are positive. Assume now that $\sigma_j(y) >0$. Then the factor $(1-3\sigma_j(y))$ belongs to the interval $(0,1)$, implying that the factor $(1-3\sigma_j(b))$ belongs to the interval $(1, \infty)$ and so $-\sigma_j(b) > 0$. But $-\sigma_j(b) > 0$ implies $ 1-4\sigma_j(b) > 1>0$. We assumed that $\sigma_j(y) >0$ and deduced $1-4\sigma_j(b)>0$, which proves 
\eqref{eq:either-or-inequality}. This finishes the proof of $\rho =1$, hence the 
proof of the claim made in \eqref{eq:recursion-condition-lambda_r} and thus the proof 
of Theorem \ref{thm:main-thm-spheres}. It only remains to prove Lemma 
\ref{lem:argument-formula}.
\begin{proof}[Proof of Lemma \ref{lem:argument-formula}]
We need to show that for all $g = \begin{psmallmatrix}
a & b\\
 c & d
\end{psmallmatrix} \in \SL_2(\R)$, we have
\begin{equation}\label{eq:argument-formula-equivalent}
\lim_{y \rightarrow \infty}{  \arg_{(- \pi/4, \pi/4)}{\left[((g  \cdot (iy))/i)^{1/2}\right]}} = - \tfrac{\pi }{4} \sgn( a c).
\end{equation}
Both sides of \eqref{eq:argument-formula-equivalent} are unchanged if we replace $g$ 
by $-g$, so we may assume $c \geq 0$ for the verification.  For $y >0, $ we abbreviate \[
w(y) := (g \cdot (iy))/i = \frac{aiy + b}{-cy + id} \in \mathcal{H} := \{ w \in \C \,:\, \real(w) >0 \}.
\]
In this proof, any asymptotic notation refers to taking $y \rightarrow \infty$. If $c = 0$,  then $ad = 1$ and we have
\[
w(y) =  \frac{a (iy) + b}{id}  = \frac{a^2 (iy) + ab}{i}  = a^2 y -i ab, \quad 
\text{hence} \quad \frac{\imag(w(y))}{\real(w(y))} = \frac{-b}{a y} \longrightarrow 0,
\]
which shows that that argument of $w(y)$ and hence that of $w(y)^{1/2}$, goes to zero, as claimed. If $c > 0$ and $a  = 0$, then $-bc = 1$ and we have
\[
w(y) = \frac{b}{-cy + di } = \frac{b^2}{y + db i} = \frac{b^2 y}{y^2 + (db)^2} - \frac{b^2 (db)i}{y^2 + (db)^2}, \quad \text{hence} \quad \frac{\imag(w(y))}{\real(w(y))} = \frac{-db}{ y} \longrightarrow 0,
\]
as claimed. Assume now that $c >0$ and that $a \neq 0$. Then
\[
w(y) =  \frac{1}{i} \left( \frac{a}{c}- \frac{1}{c ( c(iy) + d)} \right) = (-i)(a/c) + o(1).
\]
We deduce that
\begin{itemize}
\item if $a>0$, then $\arg(w(y)) \rightarrow -\pi/2$, hence $\arg( w(y)^{1/2}) \rightarrow - \pi/4$, as claimed.
\item if $a<0$, then $\arg(w(y)) \rightarrow \pi/2$, hence $\arg( w(y)^{1/2}) \rightarrow \pi/4$, as claimed.
\end{itemize}
This finishes the proof of \eqref{eq:argument-formula-equivalent} and thus the proof of Lemma \ref{lem:argument-formula}.
\end{proof}

\section{Group theoretic obstructions to 
interpolation}\label{sec:algebraic-obstructions}
In this section, we generalize the setting we have been studying so far in the following way. We replace the (embedded) codifferent $\sigma(\Oo_K^{\vee}) \subset \R^n$ of a totally real field $K$ and its square root $\sqrt{\sigma(\Oo_K^{\vee})} \subset \R^n$ by a general lattice $\Lambda \subset \R^n$ and its square root
\[
\sqrt{\Lambda}:= \{ (x_1, \dots,x_n) \in \R^n \,:\, (x_1^2, \dots, x_n^2) \in \Lambda \}.
\]
To motivate looking at possible Fourier uniqueness or non-uniqueness sets of this 
shape, we give below in \S \ref{sec:set-up-with-lattices} a translation of a general 
Fourier interpolation problem with uniqueness pairs of the form $(\sqrt{\Lambda_1}, 
\sqrt{\Lambda_2})$, to the problem of finding certain holomorphic functions on $\H^n$ 
having modular transformation behavior with respect to a certain subgroup 
$\Gamma(L_1, L_2) \leq \PSL_2(\R)^n$, where $L_i = 2\Lambda_i^{\vee}$. Ideally, we 
would want this group to be discrete and at the same time isomorphic to the free 
product $L_1 \ast L_2$. In Proposition \ref{prop:non-existence-of-good-lattices} 
below we show that, for $n \geq 2$, this can never happen. The results 
of~\S\ref{sec:algebraic-obstructions} will not be used elsewhere in this paper, but 
may be of independent interest and provide further context and motivation. 
\subsection{Generating series and functional equations}\label{sec:set-up-with-lattices}
Adopt the general setting of \S \ref{subsec:multi-radial-case}. Thus,  $n,d,d_1, \dots, d_n \geq 1$ are integers and $d = d_1 + \dots + d_n$. Fix two lattices $\Lambda_1, \Lambda_2 \subset \R^n$ and for $i = 1,2$, define $\Lambda_{i, +}:= \Lambda_i \cap [0, \infty)^n$. We want to know whether there exist functions $a_{\lambda}, \tilde{a}_{\mu}: \R^n \rightarrow \C$ such that for all $f \in \Ss(\R^d)^{\Or(d_1) \times \dots  \times \Or(d_n)}$ and all $x = (x_1,
\dots, x_n) \in \R^d$,
\begin{equation}\label{eq:interpolation-formula-general}
f(x_1, \dots, x_n) = \sum_{\lambda \in \Lambda_{1,+}}{a_{\lambda}(|x_1|, \dots, |x_n|) 
f(\sqrt{\lambda}}) + \sum_{\mu \in \Lambda_{2,+}}{{\tilde{a}}_{\mu}(|x_1|, 
\dots, |x_n|) \widehat{f}(\sqrt{\mu}}),
\end{equation}
where we used the notation $\sqrt{\lambda}:= ( \sqrt{\lambda_1}, \dots, \sqrt{\lambda_n})$ for $\lambda = (\lambda_1, \dots, \lambda_n)$.

Let us first \emph{assume} that such functions $a_{\lambda}, \tilde{a}_{\mu}$ exist and that, for each fixed $r = (r_1, \dots, r_n) \in [0, +\infty)^n$, they grow at most polynomially in their index parameters $\lambda \in \Lambda_1$ and $\mu \in \Lambda_2$ respectively. We consider the generating functions
\begin{equation}\label{eq:def-generating-functions}
F(z,r) = \sum_{\lambda \in \Lambda_{+,1}}{a_{\lambda}(r) e^{\pi i \sum_{i=1}^{n}{z_i 
\lambda_i}}}, \quad  \tilde{F}(z,r) = \sum_{\mu \in 
\Lambda_{2,+}}{{\tilde{a}}_{\mu}(r) e^{\pi i \sum_{i=1}^{n}{ z_i \lambda_i}}}, \quad z \in \H^n.
\end{equation}
By construction, each of these functions is holomorphic in $z$ and periodic with respect to 
 the lattices $2\Lambda_1^{\vee}$ or $2\Lambda_2^{\vee}$ respectively. Moreover, applying the formula \eqref{eq:interpolation-formula-general} to the Gaussian $f = g(z)$, as defined in \eqref{eq:def-gaussian}, shows that
\begin{equation}\label{eq:functional-equation-S}
g(z,r) = F(z,r) + ( z_1/i)^{-d_1/2} \cdots ( z_n/i)^{-d_n/2} \tilde{F}(-1/z,r).
\end{equation}
Conversely, no longer assuming the existence of $a_{\lambda}, \tilde{a}_{\mu}$ but the existence of holomorphic $2 \Lambda_1^{\vee}$-periodic functions $z \mapsto F( z, r)$ and holomorphic $2 \Lambda_2^{\vee}$-periodic  functions $z \mapsto \tilde{F}(z,r)$ satisfying suitable growth conditions, which are related via the functional equations \eqref{eq:functional-equation-S} and with Fourier expansions indexed over $\Lambda_{i, +}$ only (instead of the whole $\Lambda_{i}$), we can deduce an interpolation formula \eqref{eq:interpolation-formula-general} by making use of the following Proposition. 
\begin{proposition}\label{prop:density-of-Gaussians}
The linear span of all Gaussians $g(z)$, $z \in \H^n$ is dense in $\Ss(\R^d)^{\Or(d_1)
\times \dots \times \Or(d_n)}$
\end{proposition}
\begin{proof}
We defer this proof to the Appendix  \ref{sec:appendix-density-gaussians}, as  we will not need it for $n \geq 2$, but it seems worth recording. For $n = 1$, this is also contained in \cite[Lemma 2.2]{CKMRV}. 
\end{proof}

\subsection{Group theoretic and modular considerations}\label{subsec:group-theory}
The modular transformation properties of the generating functions $\tilde{F}$ and $F$ 
defined above are governed by a certain  subgroup of $\PSL_2(\R)^n$ acting on $\H^n$, 
depending upon the lattices $\Lambda_1, \Lambda_2 \subset \R^n$ (but not on the 
dimensions $d_j$) which we define next. In the notation of \S \ref{sec:pf-thm-2}, 
this subgroup can be thought of as the analogue of the subgroup of $\PSL_2(\Oo_K)$ 
generated by all elements $T^{2 \beta}, ST^{2 \beta}S$, $\beta \in \Oo_K$ in the case 
where $\Lambda_1 = \Lambda_2 = \sigma(\Oo_K^{\vee})$.   

Instead of working with $\PSL_2(\R)^n$, we find it more convenient to work with the isomorphic group $G:=\PSL_2(\R^n)$, where $\R^n = \R \times \cdots \times \R$ is viewed as commutative ring with component wise addition and multiplication.  For $x \in \R^n$ we define
\begin{equation}\label{eq:definition-Tx-Vx}
T^x := \begin{pmatrix}
1 &x\\
0 & 1
\end{pmatrix}, \qquad V^{x}:= \begin{pmatrix}
1 & 0\\
x & 1
\end{pmatrix} \in G,
\end{equation}
where $0 =(0, \dots, 0)$, $1 = (1, \dots, 1)$. We also define the element $S:= 
\begin{psmallmatrix} 0 & -1\\ 1 & 0 \end{psmallmatrix} \in G$, so that $ST^x S = 
V^{-x}$. For any lattice $L \subset \R^n$, we define the following subgroups 
of $G$:
\[
\Gamma_{\text{upp}}(L) := \{ T^{ x} \,:\, x \in L \} \cong L, \qquad \Gamma_{\text{low}}(L) := \{ V^{y} \,:\, y \in L \} \cong L
\]
and then, for any two lattices $L_1, L_2 \subset \R^n$, we also define the subgroup
\[
\Gamma(L_1, L_2) := \langle \Gamma_{\text{upp}}(L_1) \cup \Gamma_{\text{low}}(L_2) \rangle \leq G.
\]
The subgroup relevant to the setting described in \S \ref{sec:set-up-with-lattices} 
is then $\Gamma(L_1,L_2 )$, where $L_i = 2 \Lambda_i^{\vee} $. To explain this, let 
us suppose that we are given a cocycle $J: G \rightarrow \Hol(\H^n,\C^{\times})$ 
satisfying $J(T^{x}) = 1$ for all $x \in \R^n$ and $J(S)(z) = 
\prod_{j=1}^{n}{(z_j/i)^{d_j/2}}$. We may then define a slash action of $G$ (and its 
group algebra $\C[G]$) on functions $f$ on $\H^n$ by  $f| \gamma := J(\gamma)^{-1} 
\cdot (f \circ \gamma)$, $\gamma \in G$, similarly to \S \ref{sec:pf-thm-2}. 

In practice, it suffices that $J$ can be defined only on the subgroup generated by $\Gamma(L_1, 
L_2)$ and $S$,  but its existence is non-trivial and may not always be guaranteed, compare with \S \ref{sec:pf-thm-2}. On the other hand, when $8|d_j$ for all 
$j$, such a cocycle $J$ can be defined on the full group $G$, namely, we define
$J_{G;d_1,\dots,d_n}(g)=\prod_{j=1}^{n}(g_j')^{-d_j/4}$, where $g=(g_1,\dots,g_n)$, 
and $g_j'$
is the derivative of the M\"obius transformation~$g_j$.

Now consider the functions $F$, $\tilde{F}$ introduced in \S 
\ref{sec:set-up-with-lattices}. In what follows we will suppress the parameters 
$r \in [0, \infty)^n$ and $z \in \H^n$ from the notation. Using the slash action just 
introduced, $F$ and $\tilde{F}$ (as functions on $\H^n$) must satisfy, besides 
certain growth conditions,
\[
F|(T^{x}-1) = 0 \text{ for all }x \in L_1, \qquad \tilde{F}|(T^{y}-1) = 0 \text{ for all }y \in L_2, \qquad F + \tilde{F}|S = g,
\]
where $g$ is the Gaussian \eqref{eq:def-gaussian}. It suffices to find only $F$ such that
\[
F|(T^{x}-1) = 0 \text{ for all }x \in L_1, \qquad F|(V^{y}-1) = g|(V^{y}-1) \text{ for all }y \in L_2.
\]
Indeed, we can then define $\tilde{F}$ as $\tilde{F}=g|S-F|S$ and this function will be $L_2$-periodic. 


We  see from the above  cohomological  formalism that any relation between elements in the group $\Gamma(L_1,L_2)$ imposes a condition on the $1$-cocycle $ \Phi: \gamma \mapsto F|(\gamma-1)$. There are trivial relations that come from the fee abelian subgroups $\Gamma_{\text{upp}}(L_1)$ and $\Gamma_{\text{low}}(L_2)$ that are always respected. There  is, however, no reasons why a ``mixed" relation between elements of these two groups should hold, as  such a relation translates to non-trivial conditions for the Gaussian $g$. Thus, one would like that
\begin{equation}\label{eq:condition-free}
 \Gamma(L_1, L_2) \text{ \emph{is the free inner product of} } \Gamma_{\text{upp}}(L_1)  \text{ \emph{and}  }\Gamma_{\text{low}}(L_2). \tag{F}
\end{equation}
A natural further desideratum is:
\begin{equation}\label{eq:condition-discrete}
 \Gamma(L_1, L_2) \text{ \emph{is discrete in} } G  \cong \PSL_2(\R)^n. \tag{D}
\end{equation}
In fact, the existence of $F$ and $\tilde{F}$ with the above transformation 
properties \emph{implies} \eqref{eq:condition-free} by the following proposition. 
\begin{proposition}\label{prop:necessity-of-F}
    Assume that there exist functions $F$ and $\tilde{F}$ as 
    in~\eqref{eq:def-generating-functions}
    satisfying~\eqref{eq:functional-equation-S}. Then     
    condition~\eqref{eq:condition-free} holds.
\end{proposition}
\begin{proof}
By way of contradiction, assume that \eqref{eq:condition-free} fails and consider a 
non-trivial relation
\[
V^{y_1}T^{x_1}V^{y_2}T^{x_2} \cdots V^{y_m}T^{x_m} = 1
\]
with $m \geq 1$ minimal and with $x_1, \dots, x_m \in L_1$, $y_1, \dots, y_m \in 
L_2$, all nonzero (by conjugation with some $T^x$ or $V^y$ if necessary, we can bring 
any minimal non-trivial  relation into the above form). Consider 
the cocycle $\Phi(\gamma) = F|(\gamma-1)$ as above and apply the cocycle property $\Phi(\gamma_1 \gamma_2) = \Phi(\gamma_1)|\gamma_2 + \Phi(\gamma_2)$ repeatedly, to 
obtain
\begin{equation}
\label{eq:repeatedly-using-cocycle-property}
0 = \Phi(1) = \sum_{i=1}^{m}{\Phi(V^{y_i})|P_i} = \sum_{i=1}^{m}{(g|V^{y_i}P_i - g|P_i)} , \quad P_i := T^{x_i}V^{y_{i+1}}\cdots V^{y_m}T^{x_m}.
\end{equation}
Since $x_m \neq 0$, all $2m$ group elements $V^{y_i}P_i, P_i$ are pairwise distinct 
by minimality of $m$. Thus, we have an identity $0 = \sum_{j=1}^{2m}{\delta_j 
J(\gamma_j)^{-1}g(\gamma_j)}$ with $\delta_j \in \{ \pm 1 \}$ and 
with pairwise distinct $\gamma_j \in \Gamma(L_1, L_2)$. We obtain the 
desired contradiction by specializing this identity to some point $z \in \H^n$, which 
is not fixed by any $\gamma_i \gamma_j^{-1}$ for $i \neq j$ and invoking
Lemma~\ref{lem:linear-independence}.
\end{proof}

\begin{remark}
In the above proof we assumed the existence of the automorphy factor $J$ defined on 
the group~$\Gamma(L_1, L_2)$. When $J$ is not well-defined we can 
modify the argument to obtain the same conclusion in the following way. 
Consider the 
abstract free product $\widetilde{\Gamma} = \widetilde{\Gamma}(L_1, L_2) = 
\Gamma_{\text{upp}}(L_1)*\Gamma_{\text{low}}(L_2)$ and define 
$\widetilde{J}\colon\widetilde{\Gamma}\to \Hol(\H^n,\C^{\times})$ by
    \[
    \widetilde{J}(T^{x})(z) = 1, \qquad\qquad 
    \widetilde{J}(V^{y})(z) = \mu(T^{-y}Sz)\mu(z),
    \]
    for $x \in L_1$, $y \in L_2$, where, as in \S \ref{sec:pf-of-thm-1}, $\mu(z) = \prod_{j=1}^{n}{(z_j/i)^{d_j/2}}$. Since $\mu(z) \mu(Sz) = 1$, the cocycle $\tilde{J}$ is well-defined on $\Gamma_{\text{low}}(L_2)$, hence on all of $\tilde{\Gamma}$. Let $\pi$ denote the natural homomorphism from $\widetilde{\Gamma}$ 
onto $\Gamma(L_1, L_2)$. We may then define a right action of 
$\widetilde{\Gamma}$ on functions $f\colon\H^n\to\C$ by $f|\gamma := 
\widetilde{J}(\gamma)^{-1}(f\circ \pi(\gamma))$. 
We define a $\widetilde{\Gamma}$-cocycle $\widetilde{\Phi}$ by 
$\widetilde{\Phi}(\gamma):=F|(\gamma-1)$. 
Since $J_{G;8d_1,\dots,8d_n}\circ 
\pi$ agrees with $\widetilde{J}^8$ on the generators of~$\widetilde{\Gamma}$
we see that for all $\gamma \in \ker(\pi)$, the function $\widetilde{J}(\gamma)$ is constant and equal to some 8-th root of unity.
Then, instead of~\eqref{eq:repeatedly-using-cocycle-property}, we obtain
    \begin{equation*}
    (\widetilde{J}(R)^{-1}-1)F = \widetilde{\Phi}(R) = 
    \sum_{i=1}^{m}{\widetilde{\Phi}(V^{y_i})|P_i} = 
    \sum_{i=1}^{m}{(g|V^{y_i}P_i - g|P_i)}\,,
    \end{equation*}
where $R=V^{y_1}T^{x_1}V^{y_2}T^{x_2} \cdots V^{y_m}T^{x_m}$ is an element 
in $\ker(\pi)\subset\widetilde{\Gamma}$ with  $m \geq 1$ minimal and $P_i$ as in 
\eqref{eq:repeatedly-using-cocycle-property}. Since $F$ is $L_1$-periodic, by acting 
on 
both sides of the above equation by $T^{x}-1$ for a suitable~$x\in L_1$ (so that the 
resulting linear combination of Gaussians on the right-hand side involves $4m$  distinct elements) and again invoking Lemma~\ref{lem:linear-independence} 
for suitable $z\in\H^n$ (not fixed by any element in a finite set of non-trivial group elements) we arrive at the desired contradiction.
\end{remark}

Thus, condition~\eqref{eq:condition-free} is necessary for the existence of~$F$ 
and~$\tilde{F}$. Regarding condition~\eqref{eq:condition-discrete} we don't have a rigorous justification for its necessity. However, one can show that if \eqref{eq:condition-discrete} fails, then $\Gamma(L_1,L_2)$
contains many elliptic elements of infinite order (here we call $\gamma\in\PSL_2(\R)^n$ elliptic if each component is either an elliptic element in $\PSL_2(\R)$ or identity), and any such element $\gamma$ imposes rather strong conditions on~$\Phi(\gamma)$ and~$F$ (e.g., if the closure of $\langle\gamma\rangle$ is a maximal compact subgroup of $\PSL_2(\R)^n$, $F$ is uniquely defined by the relation $\Phi(\gamma)=F|(\gamma-1)$). Thus it seems plausible that~\eqref{eq:condition-discrete} is also necessary for the existence of $F$ and $\tilde{F}$.

Before stating the next result, let us return to the examples coming from totally real number fields. As 
already  mentioned, in the notation of \S \ref{sec:pf-thm-2},  for a 
totally real number field $K/ \Q$ of degree $n \geq 2$ and $L_1 = L_2 =  
\sigma(2\Oo_K)$, the group $\Gamma(L_1, L_2)$ is the subgroup of $\PSL_2(\Oo_K)$ generated by all elements $T^{2\beta}, V^{2 \beta}$, 
$\beta \in \Oo_K$. As such, it is well known to be a discrete subgroup of 
$G$, so \eqref{eq:condition-discrete} holds. On the other hand, the condition 
\eqref{eq:condition-free} \emph{never} holds in this case. 
%
%
To give a concrete example, if $0 \neq \beta \in \Oo_K$ is such that $1 + 5\beta \in \Oo_K^{\times}$, then
\begin{equation}
\label{eq:general-relation}
T^{-2\beta(1 + 5\beta)^{-1}}V^2 T^2 V^{2\beta}T^{-2(1 + 5\beta)^{-1}}V^{-2(1 + 5\beta)} = 1.
\end{equation}

Returning to general lattices, for  $n \geq 2$, there are unfortunately no examples of lattices $L_1, L_2 
\subset \R^n$ for which both conditions \eqref{eq:condition-discrete}, 
\eqref{eq:condition-free} hold, as the following proposition shows.

\begin{proposition}\label{prop:non-existence-of-good-lattices}
Let $n \geq 2$ and let $L_1,L_2 \subset \R^n$ be arbitrary lattices. Suppose the group $\Gamma(L_1, L_2) \leq G$ is discrete. Then \eqref{eq:condition-free} does not hold. 
\end{proposition}
\begin{proof}
Consider the following property (irreducibility) of a lattice $L \subset \R^n$
\begin{equation}\label{eq:avoiding-coordinate-axes}
L \sm \{0 \} \subset (\R^{\times})^n . \tag{I}
\end{equation}

For example, if $L$ is the image of a fractional ideal  in a totally real 
number field under the natural embedding then~\eqref{eq:avoiding-coordinate-axes} 
holds. The proof distinguishes two cases, according to whether both $L_1, 
L_2$ satisfy 
\eqref{eq:avoiding-coordinate-axes} or one of them does not. 
 

\emph{Case 1: Both $L_1, L_2$ have property 
\eqref{eq:avoiding-coordinate-axes}}. In this case, by a result of A. Selberg 
(sketched in \cite{Sel}), generalized by Benoist-Oh \cite[Cor. 1.2]{BO}, there exists 
a totally real number field~$K$ of degree~$n$ such that $\Gamma(L_1, 
L_2)$ is commensurable to a conjugate of the group $\PSL_2(\Oo_K)$ embedded 
in~$G$. Since Hilbert modular groups of totally real number fields are known to be 
irreducible lattices\footnote{We use the definition that a lattice $\Gamma$ in a 
connected, real semi-simple Lie group~$G$ with finite center is \emph{irreducible} 
if for all non-discrete closed normal subgroups~$N$ of~$G$ the subgroup $\Gamma N$ is 
dense in~$G$. The set of such irreducible lattices in~$G$ is closed under the 
equivalence relations given by conjugation and commensurability.  See \cite[\S 4.3]{Mor}. 
} in $\PSL_2(\R)^n$, it follows that $\Gamma(L_1, L_2)$ is an irreducible lattice in $G \cong \PSL_2(\R)^n$. Margulis' normal subgroup theorem \cite[Thm 4.9]{Marg} then implies that the abelianization of $\Gamma(L_1, L_2)$ must be finite, which rules out \eqref{eq:condition-free}.


\emph{Case 2: One of the lattices $L_1, L_2$ does not have property 
\eqref{eq:avoiding-coordinate-axes}}. Let us first suppose that $L_1$ does not 
have property \eqref{eq:avoiding-coordinate-axes}. Fix a nonzero element $x_0 
\in L_1$ whose (say) first coordinate is zero. We will construct a sequence of 
lattice vectors $y_k \in L_2 \sm \{0 \}$ such that the commutators 
\[
[T^{x_0}, V^{y_k}] = T^{x_0} V^{y_k}T^{- x_0} V^{- y_k} \in \Gamma(L_1, L_2)
\]
tend to $1 \in G$, as $k \rightarrow \infty$.  As we are assuming that $\Gamma( 
L_1,L_2)$ is discrete, the sequence must be stationary and so 
\eqref{eq:condition-free} would not hold.  To produce the sequence $y_k$, we apply 
Minkowski's lattice point theorem to the convex, compact, centrally symmetric bodies
\[
C_k  :=  \left \{  (t_1, \dots, t_n)  \in \R^n \,:\, |t_1| \leq 1 + k^{n-1}2^{n} \covol{( L_2)},\, \max_{2 \leq j \leq n}{|t_j|} \leq 1/k  \right \}, 
\]
whose volumes are $> 2^n \covol{( L_2)}$. We may thus choose $0 \neq y_k \in L_2 \cap C_k$ and with this choice, we  have $[T^{x_0}, V^{y_k}] \rightarrow 1$ as $k \rightarrow \infty$.

Finally, if $L_2$ does not have property \eqref{eq:avoiding-coordinate-axes}, we can modify the argument just given in an obvious way, by taking a fixed nonzero element $y_0 \in L_2$ with some vanishing coordinate and a sequence of nonzero lattice vectors $x_k \in L_1$ all of whose coordinates tend to zero, except in the coordinate where $y_0$ is zero. 
\end{proof}
To summarize the general results of this section, we have 
shown that for $n \geq 2$ and for any two lattices $\Lambda_1, \Lambda_2 \subset 
\R^n$, the following holds: If the group $\Gamma(2 \Lambda_1^{\vee}, 2 
\Lambda_2^{\vee})$ is discrete, then no interpolation formula as in 
\eqref{eq:interpolation-formula-general} can exist, by Proposition 
\ref{prop:non-existence-of-good-lattices} and Proposition \ref{prop:necessity-of-F}. 
Combined with 
our previous remark on the necessity of \eqref{eq:condition-discrete}, it may well be 
that no interpolation formula of the form \eqref{eq:interpolation-formula-general} 
exists for $n \geq 2$.

\section{Interpolation result via Hecke groups with infinite covolume}\label{sec:interpolation-result}
By the analysis of \S \ref{sec:algebraic-obstructions}, Fourier interpolation for 
square roots of lattices seems to be limited to the case of 
radial Schwartz functions and $1$-dimensional lattices. Let us revisit this case in 
more detail and compare it to similar known results on Fourier uniqueness.

Let $\alpha, \beta > 0$ and consider the one-dimensional lattices $\Lambda_1 = (1/ 
\alpha) \Z \subset \R$ and $ \Lambda_2 = (1/ \beta) \Z \subset \R$. Thus,  we 
consider (the possibility of existence of) interpolation formulas of the form
\begin{equation}\label{eq:redundnat-interpolation-formula}
f(x) = \sum_{n=0}^{\infty}{f(\sqrt{n/ \alpha})a_n(|x|)} +  \sum_{n=0}^{\infty}{ 
\hat{f}(\sqrt{n/ \beta})\tilde{a}_n(|x|)}, \quad \quad  f \in \Ss_{\text{rad}}(\R^d),\, 
x \in \R^d.
\end{equation}
The relevant subgroup of $\PSL_2(\R)$ is thus
\[
\Gamma( 2(\tfrac{1}{\alpha}\Z)^{\vee}, 2 (\tfrac{1}{\beta}\Z)^{\vee}) = \Gamma(2 
\alpha \Z, 2 \beta \Z) = \langle T^{2 \alpha}, V^{2 \beta} \rangle = \langle T^{2 
\alpha} , ST^{2 \beta} S \rangle.
\]
Conjugating the group  by $\begin{psmallmatrix}
t & 0\\
0 & t^{-1}
\end{psmallmatrix}$ with $t = (\beta/ \alpha)^{1/4}$  allows us to reduce to the case 
$\alpha = \beta$. Alternatively, we may reduce to that case by directly applying a 
scaling argument to \eqref{eq:redundnat-interpolation-formula}. We then write 
$\alpha = \beta = \lambda/2$ for $\lambda >0$ and consider the groups
\begin{equation}\label{eq:def-Gamma-lambda-H-lambda}
\Gamma(\lambda) := \Gamma(\lambda \Z, \lambda \Z) = \langle T^{\lambda}, V^{\lambda} \rangle \triangleleft H(\lambda) := \langle S, T^{\lambda} \rangle \leq \PSL_2(\R). 
\end{equation}
The latter groups $H(\lambda)$ are well-studied and known to be  discrete precisely 
when $\lambda \geq 2$ or  $\lambda = 2 \cos(\pi/p)$ for some integer $p \geq 3$ (we 
refer to~\cite{BK} or~\cite{H1} for background). The group $\Gamma(\lambda)$ is known 
to be discrete \emph{and} free precisely when $\lambda \geq 2$. The papers~\cite{RV}, 
\cite{BRS}, \cite{S} focus on the case $\lambda =2$. Recently, Sardari~\cite{Sar} 
investigated the case $1 < \lambda < 2$ to answer a question raised in~\cite{CKMRV}.  
The paper~\cite{CKMRV} itself considers the case $\lambda=1$, but in a vector-valued 
setting. 
\begin{figure}[h]
    \centering
    \begin{tikzpicture}
    \definecolor{cv0}{rgb}{0.95,0.95,0.95}
    \definecolor{cv1}{rgb}{0.90,0.90,0.90}
    \clip(-9,-0.7) rectangle (7,3.5);

    \begin{scope}[scale=0.6, xshift=-1.5cm]
    \draw[lightgray] (-6,0) -- (6,0);
    \fill[color=cv0]  (5,6.5) -- (5,0) -- (3,0) arc (0:180:3) -- (-5,0) -- (-5,6.5);
    \draw[gray,dashed] (-5,0)  --  (-5,6.5);
    \draw[gray,dashed] (5,0)  --  (5,6.5);
    \draw[gray,dashed] (-3,0) arc  (180:0:3);
    \draw (0,4.2) node[above]{$\mathcal{F}_{\lambda}$};
    \fill[black] (-5,0) circle (0.07) node[below] {$-\tfrac{\lambda}{2}$};
    \fill[black] (-3,0) circle (0.07) node[below] {$-1$};
    \fill[black] (0,0) circle (0.07) node[below] {$0$};
    \fill[black] (3,0) circle (0.07) node[below] {$1$};
    \fill[black] (5,0) circle (0.07) node[below] {$\tfrac{\lambda}{2}$};
    \end{scope}
    \end{tikzpicture}
    
    \caption{A fundamental domain for $H(\lambda)$ for $\lambda>2$}\label{fig:Hlambdadomain}
\end{figure}
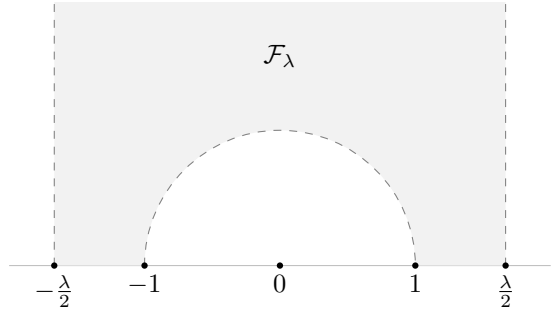

In view these results and of the conditions \eqref{eq:condition-discrete}, 
\eqref{eq:condition-free}, it remains to consider the case $\lambda > 2$, which is 
the purpose of this section. Using a series construction similar to the construction 
of Poincar{\'e} series and analogous to the one used in \cite{S}, we will prove the 
following theorem.


\begin{theorem}\label{thm:interpolation-theorem-lambda-bigger-2}
Let $\lambda \geq 2$ be a real number and let $d \geq 5$ be an integer. Set $k = d/2$. There exist sequences of entire even functions $a_{k,\lambda, n}, \tilde{a}_{k, \lambda, n} :  \C \rightarrow \C$, $n = 1,2,3, \dots$, such that for all $f \in \mathcal{S}_{\text{rad}}(\R^d)$ and all $x \in \R^d$ we have
\begin{equation}\label{eq:interpolation-formula-in-thm}
f(x) = \sum_{n=1}^{\infty}{a_{k, \lambda,n}(|x|) f( \sqrt{2n/ \lambda }) } + \sum_{n=1}^{\infty}{{\tilde{a}}_{k, \lambda, n}(|x|) \widehat{f}( \sqrt{2n/ \lambda }) }
\end{equation}
and both series converge absolutely and uniformly on $\R^d$. There are absolute constants $C_1,C_2, C_3 > 0$ such that
\begin{equation}\label{eq:precise-uniform-upper-upper-bound-in-thm}
\sup_{r \in \R}{|a_{k, \lambda,n}(r)|} + \sup_{r \in \R}{|\tilde{a}_{k, \lambda,n}(r)| } \leq C_1 (C_2/ k)^{k/2} n^{k}
\end{equation}
for all $n \geq 1$ and such that
\begin{equation}\label{eq:precise-point-wise-upper-bound-in-thm}
|a_{k, \lambda,n}(r)| +|\tilde{a}_{k, \lambda,n}(r)| \leq C_3 n^{k/2+9/8} r^{-k+9/4}
\end{equation}
for all $ r > 0$ and $n \geq 1$.
\end{theorem}

\begin{corollary}\label{cor:non-radial-uniqueness-corollary-lambda-bigger-two}
For all $d \geq 5$ and $\alpha, \beta > 0$ such that $\alpha \beta > 1$ and all integers $n_0 \geq 1$,  the pair 
$\left(\cup_{n \geq n_0}{\sqrt{n/\alpha}S^{d-1}}, \cup_{n \geq 
n_0}{\sqrt{n/\beta}S^{d-1}} \right)$ is  a Fourier uniqueness pair for $\mathcal{S}(\R^d)$. If $\alpha \beta =1$, then this is true for $n_0 =1$.  
\end{corollary} 
\begin{remark}
In the known case $\lambda =2$, one can actually take $n_0 = \lfloor (d + 4)/8 \rfloor$ for all $d \geq 2$ in the statement of the corollary. This follows from \cite{BRS}, see also \cite[Thm. 1 and 2]{RS}.
\end{remark}
\begin{proof}[Proof of Corollary \ref{cor:non-radial-uniqueness-corollary-lambda-bigger-two}]
As already explained, it suffices to consider $\alpha = \beta = \lambda/2$. For radial functions and $n_0 = 1$, this is a direct consequence of Theorem~\ref{thm:interpolation-theorem-lambda-bigger-2}. For $n_0 >1$, we will prove in Appendix~\ref{sec:appendix-removal-nodes} that  the radial interpolation formula \eqref{eq:interpolation-formula-in-thm} can be modified, so that both series start at $n = n_0$. We will see that this is a consequence of the fact that, for $\lambda >2$, the space modular forms of weight $k$ on $H(\lambda)$ is infinite dimensional; see \cite[\S 3]{H1} or also \cite[Ch. 4]{BK}.
Finally, the case of general Schwartz function follows from the case of radial Schwartz functions by \cite[Cor 2.2]{S}. 
\end{proof}
The proof of Theorem \ref{thm:interpolation-theorem-lambda-bigger-2} will occupy the 
remainder of \S \ref{sec:interpolation-result}. We remark that  weaker and less 
explicit bounds than \eqref{eq:precise-uniform-upper-upper-bound-in-thm} and 
\eqref{eq:precise-point-wise-upper-bound-in-thm} suffice to establish 
\eqref{eq:interpolation-formula-in-thm} with point-wise absolute convergence. The 
more explicit bounds \eqref{eq:precise-uniform-upper-upper-bound-in-thm} and 
\eqref{eq:precise-point-wise-upper-bound-in-thm} can be used to upgrade the 
uniqueness result of Corollary 
\ref{cor:non-radial-uniqueness-corollary-lambda-bigger-two} to an 
\emph{interpolation} formula that can be written as in \cite[Thm.~1]{S} which then 
justifies the claim made in \S \ref{subsec:general-lattices-in-intro}. But even 
without these more explicit bounds,  \cite[Cor 2.1]{S} allows to deduce \emph{some} 
interpolation result from \eqref{eq:interpolation-formula-in-thm}, but possibly 
suboptimal from analytic point of view. 
The details of this passage are almost identical to 
the analysis in \cite[\S 3]{S} so we will not give them here. However, we include a 
proof of~\eqref{eq:precise-uniform-upper-upper-bound-in-thm}
and~\eqref{eq:precise-point-wise-upper-bound-in-thm} 
in Appendix~\ref{sec:appendix-explicit-upper-bounds}, for the sake of completeness and since it is not 
obvious how to 
generalize the corresponding proof in the case for $\lambda = 2$ from~\cite{S}. 
In this section \S\ref{sec:interpolation-result}, we will prove enough 
to establish~\eqref{eq:interpolation-formula-in-thm} with absolute uniform 
convergence by proving a version 
of~\eqref{eq:precise-uniform-upper-upper-bound-in-thm} with unspecified 
dependence on~$k$ and~$\lambda$.
%

\subsection{Preliminaries for the proof of Theorem 
\ref{thm:interpolation-theorem-lambda-bigger-2}}
Below in~\S\ref{sec:def-of-F-and-tilde-F}, we will define the functions $a_{k, 
\lambda, n}(r)$, 
$\tilde{a}_{k, \lambda, n}(r)$ that enter \eqref{eq:interpolation-formula-in-thm} in 
Theorem \ref{thm:interpolation-theorem-lambda-bigger-2} as the Fourier coefficients 
of certain $2$-periodic holomorphic function $F_{k, \lambda}(z,r)$, $\tilde{F}_{k, 
\lambda}(z,r)$ but before defining those, we gather here some  notation and 
preliminary results. 
\subsubsection{Notation for \S \ref{sec:interpolation-result}}

For the remainder of \S \ref{sec:interpolation-result}, $k$ denotes a real number and 
we will also assume (most of the time) that $k > 2$. 
For $x \in \R$ we use the elements $T^{x}, 
V^{x} \in \PSL_2(\R)$ as defined in \eqref{eq:definition-Tx-Vx} as well as the 
element $S = \begin{psmallmatrix} 0 & -1\\1 & 0 \end{psmallmatrix}$. We allow 
$\lambda$ to be a complex number and define $\Gamma(\lambda)$ and $H(\lambda)$ as 
subgroups of $\PSL_2(\C)$ via the generators as in 
\eqref{eq:def-Gamma-lambda-H-lambda}. 
The reason for this is mainly for the proof of part (vi) of Lemma 
\ref{lem:inequalities-1} below, but otherwise, we are only interested in real 
$\lambda \geq 2$.
For expediency, we will sometimes use the notation $\gamma= \begin{psmallmatrix} 
a_{\gamma} & b_{\gamma}\\ c_{\gamma} & d_{\gamma} \end{psmallmatrix}$ for entries of 
a $2$-by-$2$ matrix. If~$\gamma$ is an element of $\PSL_2(\R) = \SL_2(\R)/ \{ \pm I 
\}$, we will only use such notation if the expression in terms of $a_{\gamma}, 
b_{\gamma}, c_{\gamma},d_{\gamma}$ is well-defined for $\gamma \in \PSL_2(\R)$, e.g. 
$|a_{\gamma}| \in \R_{\geq 0}$ is well-defined for $\gamma \in \PSL_2(\R)$ and so is 
the condition $c_{\gamma} \neq 0$.  
\subsubsection{Slash action}\label{subsec:slash-action}
For $\lambda \geq 2$, it is known that the only relation in the Hecke group $H(\lambda)$ is $S^2 = 1$. We may therefore define a $1$-cocycle $j_k : H(\lambda) \rightarrow \Hol(\H, \C^{\times})$ by prescribing its values on the generators $S, T^{\lambda}$:
\[
j_k(S)(z) := j_k(S, z) = (z/i)^{k} := \exp(k \log(z/i)), \quad j_k(T^{\lambda})(z) := j_k(T^{\lambda},z):=1
\]
and in general by requiring the cocycle property $j_k(\gamma_1 \gamma_2) = 
(j_k(\gamma_1) \circ \gamma_2) \cdot j_k(\gamma_2)$ to hold. Since $j_k$ respects the 
relation $S^2 = 1$, this is possible. We define a right action of $H(\lambda)$ on the 
space of all $\C$-valued functions $F$ on $\H$, by $F|_k \gamma := j_k(\gamma)^{-1} 
\cdot (F \circ \gamma)$.  We extend it to the group algebra $\C[H(\lambda)]$ in the 
usual way.
\begin{lemma}\label{lem:absolute-value-cocycle}
For all $k \in \R$, all $\gamma  \in H(\lambda)$, and all $z \in \H$ we have
\[
|j_k(\gamma)(z)| = |c_{\gamma}z + d_{\gamma}|^k.
\]
\end{lemma}
\begin{proof}
Both sides of the claimed identity are $1$-cocyles $H(\lambda) \rightarrow C(\H, 
\R_{>0})$, so it suffices to verify the identity for the generators $S, 
T^{\lambda}$ 
of $H(\lambda)$, in which case it follows from the definitions.
\end{proof}

\subsubsection{Complex $\lambda$}\label{subsec:complex-lambda}We will need the 
following lemma.
\begin{lemma}\label{lem:gamma-lambda-is-free}
For $\lambda \in \C$ with $|\lambda| \geq 2$ the group $\Gamma(\lambda) = \langle T^{\lambda}, V^{\lambda} \rangle \leq \PSL_2(\C)$ is freely generated by $T^{\lambda}$ and $V^{\lambda}$.
\end{lemma}
\begin{proof}
Consider the following subsets of $\mathbb{P}^1(\C)$:
\[
X_{\lambda} := \{ z\in \C \,:\, |z| \leq 1/(|\lambda|-1) \}, \quad Y_{\lambda} = \{ \infty \} \cup \{ z \in \C \,:\, |z| \geq (|\lambda|-1) \}.
\]
Let $m \in \Z \sm \{0 \}$. By the Ping Pong lemma, it suffices to show that
\[
T^{m \lambda}X_{\lambda} \subset Y_{\lambda}, \qquad V^{m \lambda} Y_{ \lambda} \subset X_{\lambda}.
\]
Since $S X_{\lambda} = Y_{\lambda}$, $S^2 = 1$ and $ST^{m \lambda}S = V^{-m \lambda}$, it suffices to prove the first of these containments. And indeed,  for $z \in X_{\lambda}$, we have
\[
|T^{m \lambda}z|= |z + m\lambda| \geq |m||\lambda|-|z| \geq |\lambda|-|z| \geq |\lambda|-\frac{1}{|\lambda|-1}  \geq |\lambda|-1,
\]
since the last inequality is equivalent to $|\lambda| \geq 2$.
\end{proof}
\subsubsection{Special subsets of $\Gamma(\lambda)$}\label{sec:special-subsets}
For $\lambda \in \C$ with $|\lambda| \geq 2$, we define the subset $\mathcal{V}_{\lambda} \subset \Gamma(\lambda)$ to be the set of all $\gamma \in \Gamma(\lambda)$ of the form
\begin{equation}\label{eq:general-element-from-V-lambda}
\gamma = V^{e_1 \lambda} T^{f_1 \lambda} V^{e_2 \lambda} T^{f_2 \lambda} \cdots  V^{e_n \lambda} T^{f_n \lambda},
\end{equation}
where $n \geq 1$ and $e_1, \dots, e_n, f_1, \dots f_{n-1} \in \Z \sm \{0\}$, $f_n \in 
\Z$.

We also define two subsets $\Rr_{\lambda}, \tilde{\Rr}_{\lambda} \subset \{1 \} \cup \mathcal{V}_{\lambda}$ by
\[
\Rr_{\lambda} := \{ \gamma \in \mathcal{V}_{\lambda} \,:\, \gamma \text{ as in } \eqref{eq:general-element-from-V-lambda} \text{ with } f_n = 0 \}, \quad \tilde{\Rr}_{\lambda} := \{1 \} \cup  \{ \gamma \in \mathcal{V}_{\lambda} \,:\, \gamma \text{ as in } \eqref{eq:general-element-from-V-lambda} \text{ with } f_n \neq 0 \}.
\]
The set $\mathcal{V}_{\lambda}$ is stable under right multiplication by powers of $T^{\lambda}$ and  $ \mathcal{R}_{\lambda}$ is a complete set of pairwise inequivalent representatives for $\mathcal{V}_{\lambda}/ \langle T^{\lambda} \rangle$. Similarly, $\{1 \} \cup \mathcal{V}_{\lambda}$ is stable under right multiplication by powers of $V^{\lambda}$ and $\tilde{\mathcal{R}}_{\lambda}$ is a complete set of pairwise inequivalent representatives for $(\{1 \} \cup \mathcal{V}_{\lambda})/ \langle V^{\lambda} \rangle$. 

\begin{lemma}\label{lem:inequalities-1}
Consider an element $\gamma \in \mathcal{V}_{\lambda}$ as in \eqref{eq:general-element-from-V-lambda} and write $\gamma = \begin{psmallmatrix}
a & b\\ c & d
\end{psmallmatrix}$, so that the entries $a,b,c,d$ depend on $n, e_i, f_i$ and 
$\lambda$. Then the following holds:
\begin{enumerate}[(i)]
\item If $f_n =0$, then $|c| \geq  |d|$.
\item If $f_n  \neq 0$, then $|d| \geq  |c|$.
\item $c \neq 0 \neq d$.
\item $|a| \leq |c|$ and $|b| \leq |d|$.
\item Viewing $\lambda$ as a formal variable and the entries of $\gamma$ as elements 
of $\Z[\lambda]$, the degrees of the polynomials~$c$ and~$d$ are at least $2n-2$. 
\item Viewed as functions of $\lambda \in  [2, \infty)$, the entries $|c|$ and $|d|$ are monotonically increasing on $[2, \infty)$.
\end{enumerate}
\end{lemma}
\begin{proof}
We prove parts (i), (ii), and (iii) simultaneously, using induction on $n$, by 
multiplying on the right with a non-trivial power of $V^{\lambda}$ or $T^{\lambda}$ . 
The base case is $n = 1$, $f_1 = 0$, so $\gamma =V^{e_1 \lambda}$ and 
the inequality in (i) holds trivially and certainly $c_{\gamma}d_{\gamma} \neq 0$. 
For the inductive step, assume $n \geq 2$. If  $f_n \neq 0$, set $\gamma' = \gamma 
T^{-f_n \lambda}$ and if $f_n = 0$, set $\gamma' = \gamma V^{-e_n \lambda}$. Thus, we 
have
\[
\text{either} \quad
\gamma = \gamma' T^{f_n \lambda} = \begin{pmatrix}
\ast & \ast\\
c_{\gamma'} & d_{\gamma'} + f_n \lambda c_{\gamma'}
\end{pmatrix} \quad \text{or} \quad   \gamma = \gamma' V^{e_n \lambda} = \begin{pmatrix}
\ast & \ast\\
c_{\gamma'}+  e_n \lambda d_{\gamma'} & d_{\gamma'} 
\end{pmatrix}.
\]
If $f_n \neq 0$, then  $|c_{\gamma'}| \geq |d_{\gamma'}| >0$  by inductive hypothesis and hence
\[
|d|=|d_{\gamma'} +  \lambda f_n c_{\gamma'}| \geq |f_n||\lambda||c_{\gamma'}|-|d_{\gamma'}| \geq 2|c_{\gamma'}|-|c_{\gamma'}| = |c_{\gamma'}| = |c|>0,
\]
as desired. If $f_n = 0$, then $|d_{\gamma'}| \geq |c_{\gamma'}|>0$ by inductive hypothesis and we deduce $|c| \geq |d| > 0$ in a similar way.  

Part (iv) may be proved by induction on $n$, in the reverse order, that is, by multiplying elements $\gamma$   \emph{from the left} by elements $V^{e \lambda}T^{f \lambda}$, starting with $(e, f)= (e_n, f_n)$ and $\gamma = 1$, then $(e,f) = (e_{n-1},f_{n-1})$ and so on. The proof can be given almost exactly  as in \cite[Lemma 5.2]{S}  in the case $\lambda =2$. 

Part (v) can also be proved by induction on $n$, as parts (i) and (ii). In fact, one has $\deg(c) = \deg(d) + 1$, if $f_n = 0$ and $\deg(d) = \deg(c)  +1$ if $f_n \neq 0$. 

Part (vi) is easily verified for $n = 1$. For $n \geq 2$, note that parts (v) and 
(iii) together imply that the functions $c$ and $d$ are non-constant polynomial 
functions of $\lambda$ (with coefficients in $\Z$, depending upon $e_i, f_i$), all of 
whose complex zeros lie in the disc $|\lambda| < 2$. It follows from the Gauss--Lucas 
theorem that the zeros of their first derivatives also lie in that disc. In 
particular, the derivatives of the polynomials $c$ and $d$ have no \emph{real} zeros 
in $\R \sm (-2, 2)$ and this implies the claim in (vi). 
\end{proof}
As a final preliminary fact, we record the following consequence of Lemma \ref{lem:inequalities-1}:
\begin{equation}\label{eq:a-over-c-upper-upper-bound}
\max{\left(|\gamma z|, |\gamma S z|,|Sz| \right) } \leq 1 + \imag(z)^{-1}, \quad \text{ for all } \gamma \in \mathcal{V}_{\lambda}, z \in \H.
\end{equation}
To see this, note that for $\gamma \in \mathcal{V}_{\gamma}$, we have $c_{\gamma} \neq 0$ and so
\[
|\gamma z| = \left|\frac{a_{\gamma}}{c_{\gamma}}- \frac{1}{c_{\gamma}(c_{\gamma}z + d_{\gamma})} \right| \leq 1 + |c_{\gamma}|^{-2} \imag(z)^{-1} \leq  1+ \imag(z)^{-1},
\]
where the first inequality uses part (iv) and the second uses part (vi) of of Lemma 
\ref{lem:inequalities-1} and the observation that for $\lambda=2$ (iii) implies 
$|c_{\gamma}|\ge1$. The upper bound for $|\gamma S z|$ follows in 
the same way, since $\gamma S  = \begin{psmallmatrix} b_{\gamma} &-a_{\gamma}\\
d_{\gamma} & -c_{\gamma} \end{psmallmatrix}$. 
\subsection{Definition of $a_{k, \lambda, n}$ and $\tilde{a}_{k, \lambda, n}$ via their generating functions}\label{sec:def-of-F-and-tilde-F}
For $r \in \C$ and $z \in \H$, define $\varphi_r(z) = e^{\pi i z r^2}$. Using the subsets $\mathcal{V}_{\lambda}$ defined in \S \ref{sec:special-subsets} and the slash action defined in \S \ref{subsec:slash-action}, we define, for all $\lambda \geq 2$, $k  \in \R$, $z \in \C$, $r \in \C$, the formal series
\begin{equation}\label{eq:def-F-tilde-F-as-series}
F_{k, \lambda}(z,r) := (-1)\sum_{\gamma \in \mathcal{V}_{\lambda}}{(\varphi_r|_k \gamma)(z)}, \qquad \tilde{F}_{k, \lambda}(z,r) := \sum_{\gamma \in \mathcal{V}_{\lambda} \cup \{1 \}}{(\varphi_r|_k \gamma S)(z)}.
\end{equation}
Formally, we clearly have
\begin{equation}\label{eq:S-functional-equation}
F_{k, \lambda}(z,r) + (z/i)^{-k} \tilde{F}_{k, \lambda}(-1/z,r) = \varphi_r(z).
\end{equation}
Moreover, since $\mathcal{V}_{\lambda}T^{\lambda} = \mathcal{V}_{\lambda}$ and since
 \[
(\mathcal{V}_{\lambda} \cup \{1 \}) ST^{\lambda} = ((\mathcal{V}_{\lambda} \cup \{1 \})V^{- \lambda}) S = (\mathcal{V}_{\lambda} \cup \{1 \})S,
\]
both $F_{k, \lambda}(z,r)$ and $\tilde{F}_{k,\lambda}(z,r)$ are $\lambda$-periodic in 
$z$ (at least formally). The next lemma asserts that, for $k >2$, the series 
defined in \eqref{eq:def-F-tilde-F-as-series} converge absolutely and uniformly on compacta and thus show that 
all of these formal identities hold at the level of functions. 
\begin{lemma}\label{lem:convergence}
Fix real $k >0$, $\lambda \geq 2$, $X \geq 1$, $y_0 >0$ and  fix a  compact subset $\Omega \subset \C$. Define $B := \{ z \in \H \,:\,  \imag(z) \geq y_0, |\real(z)| \leq X \}$. There is a constant $C >0$, depending only on $\Omega$ and $y_0$, so that
\begin{equation}\label{eq:uniform-upper-bound-summands}
\sup_{(z,r) \in B \times \Omega}{|(\varphi_r|_k \gamma)(z)|} \leq  C(1 + 1/y)^k(X + 1)^k(c_{\gamma}^2 + d_{\gamma}^2)^{-k/2}\, \text{ for all } \gamma \in \mathcal{V}_{\lambda} \cup \{S \} \cup \mathcal{V}_{\lambda}S.
\end{equation}
If $k >2$,  the series defined in \eqref{eq:def-F-tilde-F-as-series} converge absolutely and uniformly on $B \times \Omega $. 
\end{lemma}
\begin{proof}
Let $ \mathcal{W}_{\lambda}:= \mathcal{V}_{\lambda} \cup \{S \} \cup \mathcal{V}_{\lambda}S$. Let $\gamma \in \mathcal{W}_{\lambda}$ and $(z, r) \in B \times \Omega$. By Lemma \ref{lem:absolute-value-cocycle} and the estimate \eqref{eq:a-over-c-upper-upper-bound}, we have 
\begin{equation}\label{eq:uniform-upper-bound-summands-step-1}
|(\varphi_r|_k \gamma)(z)| = |c_{\gamma}z + d_{\gamma}|^{-k} e^{\real( \pi i (\gamma z) r^2)} \leq |c_{\gamma}z + d_{\gamma}|^{-k} e^{\pi (1 + y_0^{-1}) \sup_{r \in \Omega}{|r|^2}}.
\end{equation}
Let $A_z : \C \rightarrow \C$ denote the $\R$-linear map given by $A_z(ci + d) = cz + d$, $c,d \in \R$. Working with the operator norm of $A_{z}^{-1}$ (and using the equivalence of norms on $\End_{\R}(\C)$) we find a universal constant $C_1 >0$ so that \[
(c^2 +d^2)^{1/2}  = |A_z^{-1}(cz + d)| \leq \|A_z^{-1}\| |cz + d|\leq C_1 (X + 1)(1 + 1/y)|cz + d| \quad \text{ for all } (c,d) \in \R^2.
\]
Raising this to the power $k$ and then inserting into \eqref{eq:uniform-upper-bound-summands-step-1} yields \eqref{eq:uniform-upper-bound-summands}. For $k > 2$, uniform and absolute convergence follows now from part (vi) of Lemma \ref{lem:inequalities-1} and the fact that for $\lambda =2$, the set $\{(c_{\gamma}, d_{\gamma}\})_{\gamma \in \mathcal{W}_{\gamma}}$ is a subset of the primitive vectors in $\Z^2$ (modulo $\{ \pm 1 \}$).
\end{proof}

For each $k > 2$, $\lambda \geq 2$, $r \in \C$ and $n \in \Z$ we now define
\begin{align}
a_{k, \lambda,n}(r) := \frac{1}{\lambda}\int_{iy- \lambda/2}^{iy + \lambda/2}{F_{k, \lambda}(z,r) e^{- 2\pi i n z/ \lambda}dz}, \quad \tilde{a}_{k, \lambda,n}(r) := \frac{1}{\lambda}\int_{iy- \lambda/2}^{iy + \lambda/2}{\tilde{F}_{k, \lambda}(z,r) e^{- 2\pi i n z/ \lambda}dz}, \label{eq:def-an}
\end{align}
where $y >0$ can be taken arbitrarily, since $F$ and $\tilde{F}$ are holomorphic and $\lambda$-periodic. 
\begin{lemma}\label{lem:easy-vanishing-and-easy-upper-bound}
For $n \leq 0$, we have $a_{k, \lambda,n} = 0 =  \tilde{a}_{k, \lambda, n}$ and, as $n \rightarrow \infty$, we have
\begin{equation}\label{eq:easy-uniform-upper-bound}
\sup_{r \in \R}{|a_{k, \lambda,n}(r)|} + \sup_{r \in \R}{|\tilde{a}_{k, \lambda,n}(r)|} = O(n^{k}),
\end{equation}
where the implied constant depends only on $k$ and $\lambda$. 
\end{lemma}
\begin{proof}
We prove the assertions for $a_{k, \lambda, n}$, the ones for $\tilde{a}_{k, \lambda, n}$ are proved in the same way. Note that, by the triangle inequality,
\begin{equation}\label{eq:general-upper-bound-a-n-form-triangle-inequality}
|a_{k, \lambda,n}(r)| \leq e^{2 \pi n y/ \lambda}\sup_{|x| \leq \lambda/2}{|F_{k, \lambda}(x + iy, r)|} \quad \text{for all } y > 0.
\end{equation}
If $n \leq 0$, the exponential is bounded by $1$, while the supremum tends to $0$ as $y \rightarrow \infty$. The latter follows from Lemma \ref{lem:convergence} and its proof: we can use uniform convergence to pull the limit inside the series and the fact that $c_{\gamma} \neq 0$ for all $\gamma \in \mathcal{W}_{\lambda}= \mathcal{V}_{\lambda} \cup \{S \} \cup \mathcal{V}_{\lambda}S$. For $n \geq 1$ and $r \in \R$, we again use Lemma \ref{lem:convergence} and its proof (modified by using the trivial bound $|e^{\pi i \tau r^2}| \leq 1$ for $\tau \in \H$, $r \in \R$)  to deduce
 \[
\sup_{|x| \leq \lambda/2, r\in \R}{|F_{k, \lambda}(x + iy, r)|} \lesssim_{k, \lambda } (1 +1/y)^{k}.
 \]
This holds \emph{for all} $y>0$, in particular for $y = 1/n$, which then yields \eqref{eq:easy-uniform-upper-bound}.
\end{proof}
\subsection{Proof of \eqref{eq:interpolation-formula-in-thm} in Theorem \ref{thm:interpolation-theorem-lambda-bigger-2}}
Let $d \geq 5$ be an integer, $k = d/2$, $\lambda \geq 2$ real. We claim that the functions $a_{k, 
\lambda, n}$, $\tilde{a}_{k, \lambda,n}$ defined in \eqref{eq:def-an} are such that 
\eqref{eq:interpolation-formula-in-thm} holds. By their definition, the first 
assertion of Lemma \ref{lem:easy-vanishing-and-easy-upper-bound}, and by 
\eqref{eq:S-functional-equation}, the formula \eqref{eq:interpolation-formula-in-thm} 
holds for $f(x) = \varphi_{|x|}(z)$ for all $z \in \H$ and all $x \in \R^d$. On the 
other hand, from the bound in Lemma \ref{lem:easy-vanishing-and-easy-upper-bound},  
for fixed $x \in \R^d$, the RHS of \eqref{eq:interpolation-formula-in-thm} is 
continuous in $f \in \mathcal{S}_{\text{rad}}(\R^d)$ and so the claimed formula 
follows in general by the density of Gaussians (Proposition 
\ref{prop:density-of-Gaussians}, although we only need the case $n =1$, for which we 
can also cite \cite[Lemma 2.2]{CKMRV}). 
We prove the more precise bounds~\eqref{eq:precise-uniform-upper-upper-bound-in-thm} 
and\eqref{eq:precise-point-wise-upper-bound-in-thm} in Appendix  
\ref{sec:appendix-explicit-upper-bounds}.
\appendix

\section{Proof of the upper bounds 
\eqref{eq:precise-uniform-upper-upper-bound-in-thm} and 
\eqref{eq:precise-point-wise-upper-bound-in-thm} in Theorem 
\ref{thm:interpolation-theorem-lambda-bigger-2}}
\label{sec:appendix-explicit-upper-bounds}
We  will generalize  \cite[Lemma 5.3]{S} and then proceed similarly as in the rest of \cite[\S 5]{S}. For real $\kappa \geq 9/4$ and $\lambda \geq 2$, let us define $\tilde{\mathcal{V}}_{\lambda}:= (\{1 \} \cup \mathcal{V}_{\lambda}) S$ and
\begin{align*}
U_{\kappa, \lambda}(z) := \sum_{\gamma \in \mathcal{V}_{\lambda}}{|c_{\gamma}z + d_{\gamma}|^{- \kappa}}, \qquad 
\tilde{U}_{\kappa, \lambda}(z) :=  \sum_{\tilde{\gamma} \in \tilde{\mathcal{V}}_{\lambda}}{|c_{\tilde{\gamma}}z + d_{\tilde{\gamma}}|^{- \kappa}} = \sum_{\gamma \in \{1 \} \cup \mathcal{V}_{\lambda}}{|d_{\gamma}z-c_{\gamma}|^{-\kappa}}.
\end{align*}
We note that these are both $\lambda$-periodic, continuous functions on $\H$ because of the proof of  Lemma \ref{lem:convergence} and  because both sets $\mathcal{V}_{\lambda}$ and $\tilde{\mathcal{V}}_{\lambda}$ are stable under right multiplication by powers of $T^{\lambda}$. 
\begin{lemma}\label{lem:upper-bounds-U-tilde-U}
There is a constant $C_0 > 0$ so that for all $z = x + iy \in \H$, all $\lambda \geq 2$ and all $\kappa \geq 9/4$,
\[
\max{\left(|U_{\kappa, \lambda}(x + iy)|,|\tilde{U}_{\kappa, \lambda}(x + iy)| \right)}  \leq C_0 2^{\kappa}(y^{-\kappa/2} + y^{- \kappa}).
\]
\end{lemma}
\begin{proof}
By $\lambda$-periodicity, it suffices to consider $z = x +iy \in \H$ with $|x| \leq 
\lambda/2$. We start with the analysis of  $\tilde{U}_{k, \lambda}(x + iy)$ and 
explain the modifications for $U_{k, \lambda}$ at the end. We divide the series into 
subseries over orbits of right multiplication by $T^{\lambda}$. For this, recall from 
\S \ref{sec:special-subsets} the definition of the set 
$\tilde{\mathcal{R}}_{\lambda}$ and then note that
\begin{align*}
\tilde{U}_{\kappa, \lambda}(z) &= \sum_{ \gamma \in \tilde{ \mathcal{R}}_{\lambda}}{ \sum_{e \in \Z}{|d_{\gamma}z-(c_{\gamma}+e \lambda d_{\gamma})|^{-\kappa}}} = \sum_{ \gamma \in \tilde{ \mathcal{R}}_{\lambda}}{|d_{\gamma}|^{- \kappa} \sum_{e \in \Z}{|z-(c_{\gamma}/d_{\gamma}+e \lambda )|^{-\kappa}}}.
\end{align*}
For all $e \in \Z$, we have 
\[
|z + (c_{\gamma}/d_{\gamma}+e \lambda  )|^2 = y^2 + (\lambda e + x + c_{\gamma}/d_{\gamma})^2 \geq y^2.
\]
By part (ii) of Lemma \ref{lem:inequalities-1}, we have  $|c_{\gamma}/d_{\gamma}| \leq 1$ for all $\gamma \in \tilde{\mathcal{R}}_{\lambda}$ and therefore, for $|e| \geq 2$,
\begin{align*}
|z + (c_{\gamma}/d_{\gamma}+e \lambda  )|^2 &\geq 2y|\lambda e + x + c_{\gamma}/d_{\gamma}| 
\geq 2y \left( \lambda |e|- \lambda/2-1 \right) \\
&\geq 2 y \lambda \left( |e|- (1/2 + 1/ \lambda) \right) 
\geq 2 y \lambda (|e|-1).
\end{align*}
Using these lower bounds, we obtain
\begin{align*}
\tilde{U}_{\kappa, \lambda}(z) & \leq \sum_{\gamma \in \tilde{ \mathcal{R}}_{\lambda}}{ \frac{1}{(d_{\gamma}^2)^{\kappa/2}} \left( 3y^{- \kappa} + \sum_{|e| \geq 2}{\frac{1}{ \left( 2 y \lambda (|e|-1) \right)^{\kappa/2}}} \right)} \\
&\leq \sum_{\gamma \in \tilde{ \mathcal{R}}_{\lambda}}{ \frac{1}{((c_{\gamma}^2 + 
d_{\gamma}^2)/2)^{\kappa/2}} \left( 3 y^{- \kappa} + (2 \lambda y)^{- \kappa/2} 2 
\zeta(\kappa/2)\right)   }, 
\end{align*}
where we used that $c_{\gamma}^2 \leq d_{\gamma}^2$ for $\gamma \in \tilde{\mathcal{R}}_{\lambda}$.  Since $\kappa \geq 9/4>2$ and $\lambda \geq 2$, the claimed upper bound for $\tilde{U}_{\kappa, \lambda}$ follows by appealing once again to part (vi) of Lemma \ref{lem:inequalities-1} and arguing as at the end of the proof of Lemma \ref{lem:convergence}. 

To treat $U_{\kappa, \lambda}$, we also split the sum over $\mathcal{V}_{\lambda}$ into orbits modulo $\langle T^{\lambda} \rangle$ and use instead the set of representatives $\mathcal{R}_{\lambda}$ defined in \S \ref{sec:special-subsets} and correspondingly part (i) of Lemma \ref{lem:inequalities-1}.   
\end{proof}
Lemma \ref{lem:upper-bounds-U-tilde-U} together with the trivial bound $|e^{\pi i \tau r^2}| \leq 1$, valid for all $\tau \in \H$, $r \in \R$, implies that for \emph{all} $\lambda \geq 2$, $k \geq 5/2$, $x \in \R$, $y>0$, $r \in \R$, 
\begin{align*}
|F_{k, \lambda}(x+iy,r)| &\leq U_{k, \lambda}(x + iy) \leq C_0 2^k (y^{-k} + y^{- k/2} )
\end{align*}
for some absolute constant $C_0$, not depending on $k, \lambda, x, y$ or $r $ and that the same holds with the tilde. If we insert this into the general bound \eqref{eq:general-upper-bound-a-n-form-triangle-inequality} for integers $n \geq 1$ and set $y= \tfrac{k }{ \pi n}$ we obtain (after a short computation) \eqref{eq:precise-uniform-upper-upper-bound-in-thm} in Theorem \ref{thm:interpolation-theorem-lambda-bigger-2} for the functions $a_{k, \lambda,n}$ (the analysis for $\tilde{a}_{k, \lambda,n}$ is the same). 

To prove the remaining bound \eqref{eq:precise-point-wise-upper-bound-in-thm}, we assume that $r > 0$. Then, for some  $\beta >0$ to be determined, we write
\begin{align*}
|F_{k, \lambda}(z,r)| & \leq \sum_{\gamma \in \mathcal{V}_{\lambda}}{|c_{\gamma}z + d_{\gamma}|^{-k} e^{- \pi  \imag(\gamma z) r^2} \imag(\gamma z)^{\beta} \imag(\gamma z)^{- \beta}}\\
& \leq \sum_{\gamma \in \mathcal{V}_{\lambda}}{|c_{\gamma}z + d_{\gamma}|^{-k} \left( \frac{ \beta}{ \pi e r^2} \right)^{\beta} \imag(\gamma z)^{- \beta}} \leq \left( \frac{ \beta}{ \pi e r^2} \right)^{\beta} \imag(z)^{-\beta} U_{k-2 \beta, \lambda}(z).
\end{align*} 
We take $\beta = k/2-9/8$, so that we can apply Lemma \ref{lem:upper-bounds-U-tilde-U} with $\kappa = 9/4$ and so that
\[
|F_{k, \lambda}(z,r)|  \leq \left( \frac{\beta}{\pi e} \right)^{\beta}r^{-2 \beta} y^{-\beta} U_{9/8, \lambda}(z)  \leq  C_1 \left( \frac{\beta}{\pi e} \right)^{\beta}r^{-2 \beta}  (y^{-( \beta + 9/4)} + y^{-(9/8+ \beta)}),
\]
for some absolute constant $C_1 > 0$. We may now use this upper bound in the general 
estimate \eqref{eq:general-upper-bound-a-n-form-triangle-inequality} for integers $n 
\geq 1$ and set $y= \tfrac{\beta }{ \pi n}$ to obtain 
\eqref{eq:precise-point-wise-upper-bound-in-thm} in Theorem 
\ref{thm:interpolation-theorem-lambda-bigger-2} (the analysis for $\tilde{a}_{k, 
\lambda,n}$ is the same). 
\section{Removing finitely many interpolation nodes}\label{sec:appendix-removal-nodes}
Here, we prove the modification of Theorem \ref{thm:interpolation-theorem-lambda-bigger-2} explained in the proof of Corollary \ref{cor:non-radial-uniqueness-corollary-lambda-bigger-two}. We assume that $\lambda >2$ throughout this section. We first reformulate our problem by decomposing \eqref{eq:interpolation-formula-in-thm} into Fourier eigenspaces so that we can work with modular forms on the bigger group $H(\lambda) \supset \Gamma(\lambda)$. This is convenient since $H(\lambda)$ has only one cusp, but it requires some additional notation and preliminary explanation.

For $\epsilon  \in \{ \pm 1\}$, let $\chi_{\epsilon}: H(\lambda) \rightarrow \{ \pm 1 \}$ denote the group homomorphisms satisfying $\chi_{\epsilon}(S) = \epsilon$ and $\chi_{\epsilon}(T^{\lambda}) = 1$. We twist the slash action defined in \S \ref{subsec:slash-action} by the character $\chi_{\epsilon}$ by defining $f|_k^{\epsilon} \gamma = \chi_{\epsilon}(\gamma) j_k(\gamma)^{-1} \cdot (f \circ \gamma)$ for $\gamma \in H(\lambda)$ and functions $f$ on $\H$.  For $k \in \R$, let $M_k(\lambda, \epsilon)$ denote the space of all holomorphic functions $f: \H \rightarrow \C$ which satisfy $f|_k^{\epsilon}\gamma = f$ for all $\gamma \in H(\lambda)$ and which admit a Fourier expansion of the form $f(z) = \sum_{n=0}^{\infty}{b_n e^{\pi i (2n/ \lambda) z}}$ with polynomially growing Fourier coefficients: $b_n = O(n^c)$ for some $c = c(f)  \geq 0$.

Recall that we constructed $F_{k, \lambda}(z,r), \tilde{F}_{k, \lambda}(z,r)$ which are holomorphic and $\lambda$-periodic in $z$ and satisfy \eqref{eq:S-functional-equation}. For $\epsilon \in \{ \pm 1 \}$ we define $F_{k, \lambda}^{\epsilon}(z,r)$ by

\begin{equation}\label{eq:F-in-terms-of-B}
\begin{pmatrix}
F_{k, \lambda}^{+}\\
F_{k, \lambda}^{-}
\end{pmatrix} = \begin{pmatrix}
1 & -1\\
1 & 1
\end{pmatrix} \begin{pmatrix}
F_{k, \lambda}\\
\tilde{F}_{k, \lambda}
\end{pmatrix} \quad \Leftrightarrow \quad \begin{pmatrix}
F_{k, \lambda}\\
\tilde{F}_{k, \lambda}
\end{pmatrix} = \frac{1}{2} \begin{pmatrix}
1 & 1\\
-1 & 1
\end{pmatrix} \begin{pmatrix}
F_{k, \lambda}^{+}\\
F_{k, \lambda}^{-}
\end{pmatrix} . 
\end{equation} 
Then each $F_{k, \lambda}^{\epsilon}(z,r)$ is $\lambda$-periodic in $z$ and we have, by \eqref{eq:S-functional-equation}
\begin{equation}\label{eq:functional-equation-S-with-signs}
F^{\epsilon}_{k, \lambda}(\cdot, r)|_k^{\epsilon}(1-S) = \varphi_r|_k^{\epsilon}(1-S).
\end{equation}
In fact, \eqref{eq:functional-equation-S-with-signs} and \eqref{eq:S-functional-equation} are equivalent.
For $n \in \Z$ we define (note the sign change) \[
b_{k, \lambda,n}^{\epsilon}(r) := \frac{1}{\lambda} \int_{iy- \lambda/2}^{iy + \lambda/2}{F_{k, \lambda}^{-\epsilon}(z,r) e^{- \pi i (2n/ \lambda)z}dz},
\]
so that, by \eqref{eq:F-in-terms-of-B}, we have 
\begin{equation}\label{eq:b-in-terms-of-a}
\begin{pmatrix}
b_{k, \lambda}^{+}\\
b_{k, \lambda}^{-}
\end{pmatrix} = \begin{pmatrix}
1 & 1\\
1 & -1
\end{pmatrix} \begin{pmatrix}
a_{k, \lambda}\\
\tilde{a}_{k, \lambda}
\end{pmatrix} \quad \Leftrightarrow \quad \begin{pmatrix}
a_{k, \lambda}\\
\tilde{a}_{k, \lambda}
\end{pmatrix} = \frac{1}{2} \begin{pmatrix}
1 & 1\\
1 & -1
\end{pmatrix} \begin{pmatrix}
b_{k, \lambda}^{+}\\
b_{k, \lambda}^{-}
\end{pmatrix}.
\end{equation} 
For any $\phi^{\epsilon} \in M_k(\lambda, \epsilon)$ we can replace $F_{k, \lambda}^{\epsilon}(z,r)$ by $F_{k, \lambda}^{\epsilon}(z,r)-\phi^{\epsilon}(z)$ and these functions will still be holomorphic, $\lambda$-periodic, satisfy the functional equation \eqref{eq:functional-equation-S-with-signs} and have polynomially bounded Fourier coefficients. In particular, we can take for $\phi^{\epsilon}$ any linear combination of functions $b_{k,n}^{-\epsilon}(r)\phi_n^{\epsilon}(z)$ for suitable $\phi_n^{\epsilon}$. We may then redefine $F, \tilde{F}$ in terms of such modified $F^{+}, F^{-}$ via \eqref{eq:F-in-terms-of-B} and the (new) Fourier coefficients of $F$ and $\tilde{F}$ will still satisfy the interpolation formula \eqref{eq:interpolation-formula-in-thm} with uniform convergence. 

Thus, the proof of the remaining part of Corollary \ref{cor:non-radial-uniqueness-corollary-lambda-bigger-two} is reduced to the  proof of the following proposition. Indeed, given any integer $N \geq 1$, we can use the functions $f_n^{\epsilon}$, $1 \leq n \leq N$ provided by the proposition and linearly combine them to create  $\phi_n^{\epsilon} \in M_k(\lambda, \epsilon)$ such that $\widehat{\phi_n^{\epsilon}}(0) = 0$ and $\widehat{\phi_n^{\epsilon}}(m) = \delta_{n,m}$ for all $n,m \in \{1, \dots, N \}$ (where the hat-notation means Fourier coefficient).
\begin{proposition}\label{prop:cusp-forms}
Fix $k > 0$, $\lambda >2$ and $\epsilon \in \{ \pm 1 \}$. Then, for every integer $n \geq 1$, there exists $f_n^{\epsilon} \in M_k(\lambda, \epsilon)$ vanishing to order exactly $n$ at infinity.
\end{proposition}
Proposition \ref{prop:cusp-forms} is essentially due to Hecke \cite[\S 3]{H1} who proved the existence of such $f_n^{\epsilon}$ for all integers $n \geq k/2$. We will add a further observation (below near \eqref{eq:equality-of-conformal-maps}) to his proof and show that the construction extends to all $n \geq 1$. Hecke's treatment in loc. cit. is somewhat brief and we refer to~\cite[Chapter~4]{BK} for more details and explanation, also for parts of the proof given below.
\begin{proof}[Proof of Proposition \ref{prop:cusp-forms} ]
Let $B_1 := \{ z \in \C \,:\, - \lambda/2 < \real(z) <0, |z| >1 \}$, so that $B_1 
\cap \H$ is the left half of the fundamental domain~$\mathcal{F}_{\lambda}$ drawn in 
Figure~\ref{fig:Hlambdadomain}. Consider the following pieces of the boundary 
of~$B_1$:
 \[
L_1 = - \lambda/2 + i \R, \quad L_2 = i[1, \infty), \quad L_3 = \{ z \in \C \,:\, \real(z) < 0, |z| = 1 \}, \quad L_4 = i(- \infty, -1].
\]
By the Riemann mapping theorem, there exists a biholomorphic map $h : B_1 \rightarrow \H$. It may be chosen uniquely so that it extends continuously to the boundary of $B_1$ (minus the point $-i$), maps the latter to $\R$ and satisfies
\begin{equation}\label{eq:values-of-h-at-boundary-points}
h(i) =0, \quad h(-i) = -\infty, \quad h(i \infty)  = 1, \quad h(-i \infty) = a_0,
\end{equation}
for some $a_0 > 1$, where the values at $\pm i \infty$ are understood in the limit $\imag(\tau) \rightarrow \pm \infty$. We then have \[
h(L_1) = (1, a_0),  \quad  h(L_2) = [0, 1),  \quad h(L_3) = (- \infty, 0) \quad  \text{and}  \quad   h(L_4) = (a_0, \infty).
\] 
By the Schwarz reflection principle applied to $L_1$, $L_2$, and $L_3$, one may extend $h$ to an analytic function on $\C$ minus the set of points equivalent to $-i$ under the reflections just mentioned. Then $h|_{\H}$ is bounded, $H(\lambda)$-invariant and never takes the value $1$. 

We claim that there is $\delta > 0$ so that for all $\tau \in B_1$ with $|\imag(\tau)| \leq 2$ we have $|g(\tau)-1| \geq \delta$. To prove this, it suffices to show that for all $\tau \in B_1$ we have
\begin{equation}\label{eq:equality-of-conformal-maps}
\overline{h(\overline{\tau})} = \frac{a_0}{h(\tau)},
\end{equation}
because if we specialize the above to $\tau \in \R \cap B_1$, we get $|h(\tau)|^2 = a_0 > 1$ and can then use continuity of $h$ to prove the claim.  To prove \eqref{eq:equality-of-conformal-maps}, we note that both sides define biholomorphic mappings $B_1 \rightarrow  \{ z \in \C \,:\, \imag(z) < 0 \}$  and that they extend in the same way to the boundary points $\tau =  \pm i, \pm i \infty$.
Now Hecke proves the existence of a holomorphic function $H: \H \rightarrow \C$ satisfying
\[
h(\tau + \lambda) = H(\tau), \quad H(-1/ \tau) = -H(\tau), \quad H(\tau)^2 = h(\tau)
\]
for all $\tau \in \H$ and then  considers 
\[
g(\tau) := \frac{h'(\tau)}{H(\tau)(h(\tau)-1)},
\]
which is holomorphic and nowhere vanishing on $\H \cup \{i \infty\}$ and transforms like a modular form in $M_2(\lambda, +1)$ (we again refer to \cite[Ch. 4]{BK} for justification and details).  Using a suitable logarithm of $g$, Hecke defines  \[
f_n^{\epsilon}(\tau):= H(\tau)^{\tfrac{1-\epsilon}{2}}g(\tau)^{k/2} (h(\tau)-1)^n
\]
and proves that $f_n^{\epsilon} \in M_k(\lambda, \epsilon)$ for $n \geq k/2$. Note that since $h(\tau)-1$ vanishes to order $1$ at $i \infty$ while $H$ and $g$ are non-vanishing at $i \infty$, each $f_n^{\epsilon}$ indeed vanishes to order exactly $n$ at $i \infty$. It remains to be shown that $f_n^{\epsilon}$ belongs to $M_k(\lambda,\epsilon)$ \emph{for all} $n \geq 1$. For this,  it suffices to show that the $H(\lambda)$-invariant, continuous function $|f_n^{\epsilon}(\tau)|\imag(\tau)^{k/2}$ is bounded on the fundamental domain $\mathcal{F}_{\lambda}$.

For $\tau \in \mathcal{F}_{\lambda}$ with $\imag(\tau) \geq 2$ we have $g(\tau)-1 = O(e^{-(2 \pi/ \lambda) \imag(\tau)})$  while $g(\tau)^{k/2}$ and $H(\tau)$ are both $O(1)$. For $\tau \in \mathcal{F}_{\lambda}$ with $\imag(\tau) \leq 2$, we write
\begin{equation}\label{eq:use-of-f-star}
|f_n^{\epsilon}(\tau)|\imag(\tau)^{k/2} = |h(\tau)|^{\frac{1-\epsilon}{4}}|h(\tau)-1|^{n-k/2}|f^{\ast}(\tau)|^{k/2}, \quad \text{where} \quad f^{\ast}(\tau) := \imag(\tau)|h'(\tau)/H(\tau)|.
\end{equation}
The function $f^{\ast}: \H \rightarrow \R_{\geq 0}$ is easily seen to be bounded (see \cite[Ch. 4, p. 31]{BK}) and since we showed that $|h(\tau)-1|$ is bounded away from zero for $\tau \in \mathcal{F}_{\lambda}$ with $\imag(\tau) \leq 2$ and since we know that $h$ is bounded on $\H$,  we are done. 
\end{proof} 
\section{Proof of Proposition \ref{prop:density-of-Gaussians}}\label{sec:appendix-density-gaussians}
Recall that $n \geq 1$ and $d_j \geq 1$ are integers such that $d = d_1 + \dots + d_n$ and that we view $\R^d  =\prod_{j=1}^{n}{\R^{d_j}}$. Abbreviate $H := \Or(d_1) \times \dots \times \Or(d_n) \hookrightarrow \Or(d)$. We need to show that the linear span $W \subset \Ss(\R^d)^H$,  of all Gaussians 
\[
g(z)(x)  = e^{\pi i \sum_{j=1}^{n}{ z_j |x_j|^2}}, \quad z \in \H^n, x_j \in \R^{d_j},
\]
is dense in $\Ss(\R^d)^{H}$. As a matter of notation, we will write $g(z)(x)  =
g(z,x) = g_z(x)$ in this proof.

By adapting the proof of the fact that $C_{c}^{\infty}(\R^n)$ is dense in
$\Ss(\R^n)$, one may show that $C_c^{\infty}(\R^d)^{\Or(d)}$ is dense in
$\Ss(\R^d)^{H}$. In particular the larger space $C_c^{\infty}(\R^d)^{H}$ is dense in
$\Ss(\R^d)^{H}$.


We now fix $f \in C_c^{\infty}(\R^d)^H$ and aim to show that $f \in \overline{W}$. 
Fix positive reals $b_1, \dots, b_n > 0$ and consider the function
 \[
h(x) := f(x) e^{\pi \sum_{j=1}^{n}{b_j |x_j|^2}} , \quad x = (x_1, \dots, x_n) \in \R^d, \quad x_j \in \R^{d_j}.
 \]
Then $h \in C_c^{\infty}(\R^d)^H$. We claim that there exists a function $\eta \in 
C_c^{\infty}(\R^n)$ such that
\begin{equation}\label{eq:h-in-terms-of-eta}
h(x) = \eta(|x_1|^2, \dots, |x_n|^2)  \quad \text{for all } x \in \R^d.
\end{equation}
To prove this, let us fix the unit vectors $e_j \in S^{d_j-1} \subset \R^{d_j}$ and 
define $h_0 \in C_c^{\infty}(\R^n)^{\Or(1) \times \cdots \times
\Or(1)}$ by $h_0(t_1, \dots, t_n) := h(t_1 e_1, \dots, t_n e_n)$. Since the algebra 
of real polynomials in $n$ variables, which are even in each variable, is generated 
(as an algebra) by the squares of the variables, a general result of G.~Schwarz 
\cite{Sch}  implies the existence of $\eta \in C_c^{\infty}(\R^n)$ such that 
$h_0(t_1, 
\dots, t_n) = \eta(t_1^2, \dots, t_n^2)$ for all $t_j \in \R$ and this function then 
also satisfies \eqref{eq:h-in-terms-of-eta}.

Now, for a function $u \in \Ss(\R^n)$ such that $\widehat{u}$ is compactly
supported, but otherwise unspecified for the moment, we write
\begin{align*}
f(x) &= h(x) g_{(ib_1, \dots, ib_n)}(x) \\
     &= (\eta-u)(|x_1|^2, \dots, |x_n|^2)g_{(ib_1, \dots, ib_n)}(x) + u(|x_1|^2, \dots, |x_n|^2)g_{(ib_1, \dots, ib_n)}(x)\\
     &= (\eta-u)(|x_1|^2, \dots, |x_n|^2)g_{(ib_1, \dots, ib_n)}(x) + \int_{\R^n}{\widehat{u}(\xi) g_{(ib_1 + 2 \xi_1, \dots, ib_n + 2 \xi_n)}(x) d \xi},
\end{align*}
where we applied the Fourier inversion on~$\R^n$ in the last step. The latter integral
belongs to~$\overline{W}$, regardless of the choice of~$u$, as long as
$\widehat{u}$ has compact support. This follows from integration
theory in Fr{\'e}chet spaces and continuity of the map $\H^n \rightarrow \Ss(\R^d)$,
$z \mapsto g_z$ (or alternatively by approximation via Riemann sums). It therefore suffices to show that the term involving $\eta-u$ can be made
arbitrarily small in the Schwartz topology. To see this, consider the linear map
$E:\Ss(\R^n) \rightarrow \Ss(\R^d)^{H}$, defined by $E \varphi(x) := \varphi(|x_1|^2,
\dots, |x_n|^2)$. It continuous for the Schwartz topology and multiplication by
$g_{(ib_1, \dots, i b_n)}$ is continuous. Since the space of $u \in \Ss(\R^n)$
such that $\widehat{u}$ has compact support is dense in $\Ss(\R^n)$ and~$E$ is
continuous, we can choose $u$ in such a way that $E(\eta-u)$ is in any
prescribed open zero neighborhood of $\Ss(\R^d)$. This finishes the proof of
Proposition \ref{prop:density-of-Gaussians}.

\end{document}